\documentclass[10pt]{amsart}
\usepackage{graphicx,amssymb,amsfonts,amsmath,amsthm,newlfont}
\usepackage{epsfig,url}
\usepackage{color}

\usepackage[all,2cell]{xy} \UseAllTwocells \SilentMatrices

\newtheorem{thm}{Theorem}[section]
\newtheorem{cor}[thm]{Corollary}
\newtheorem{lem}[thm]{Lemma}
\newtheorem{prop}[thm]{Proposition}
\theoremstyle{definition}
\newtheorem{defn}[thm]{Definition}

\newtheorem{example}[thm]{Example}

\theoremstyle{remark}
\newtheorem{rem}[thm]{Remark}

\numberwithin{equation}{section}

\newcommand{\Z}{\mathbb Z}
\newcommand{\C}{\mathbb C}

\newcommand{\R}{\mathbb R}

\newcommand{\Lef}{\mathbb{L} }

\newcommand{\GE}{\mathbb{G} }

\newcommand{\XX}{\mathfrak{X}}

\newcommand{\M}{\mathcal{M}}

\newcommand{\ssl}{\mathfrak{sl}}

\newcommand{\MI}{\mathcal{MI}}

\newcommand{\llam}{\underline{\lambda}}

\newcommand{\Q}{\mathbb Q}

\newcommand{\s}{\mathbf{s}}

\newcommand{\To}{\longrightarrow}

\newcommand{\SL}{\mathrm{SL}}

\newcommand{\Or}{\mathcal{O}}

\newcommand{\V}{\mathcal{V}}

\newcommand{\VV}{\mathbb{V}}

\newcommand{\mm}{\mathfrak{m} }

\newcommand{\HH}{\mathfrak{H} }

\newcommand{\id}{\mathrm{id} }

\newcommand{\LL}{\mathbb{L}}

\newcommand{\sfh}{\mathsf{h}}
\newcommand{\sfw}{\mathsf{w}}

\newcommand{\HM}{\mathcal{HM}}

\newcommand{\comp}{\mathrm{comp}}

\begin{document}
\author{Francis Brown}
\begin{title}[Real analytic cusp forms  for $\mathrm{SL}_2(\mathbb{Z})$] { A class of non-holomorphic modular forms III: Real analytic cusp forms for $\mathrm{SL}_2(\mathbb{Z})$}\end{title}
\maketitle
\begin{abstract} We define canonical real analytic versions of modular forms of integral weight for the full modular group,  generalising real analytic Eisenstein series.  They are  harmonic Maass waveforms with poles at the cusp, whose   Fourier coefficients  involve  periods and quasi-periods of  cusp forms, which are conjecturally transcendental.   In particular, we settle the question of finding  explicit  `weak harmonic lifts' for every eigenform of integral weight  $k$ and level one. We  show that mock modular forms of integral weight are algebro-geometric and have Fourier coefficients proportional  to $n^{1-k}(a'_n + \rho a_n)$ for $n\neq 0$, where $\rho$ is the normalised permanent of the period matrix of the corresponding motive, and $a_n, a'_n$ are the Fourier coefficients of a Hecke eigenform and a weakly holomorphic Hecke eigenform, respectively.  More generally, this  framework provides a conceptual explanation for the  algebraicity of the coefficients of mock modular forms in the CM case. 
\end{abstract}

\section{Introduction}
Let $\HH$ denote the upper-half plane with the usual left action by $\Gamma = \SL_2(\Z)$. This paper is the third in a series \cite{ZagFest},  \cite{EqEis}  studying subspaces of the vector space $\mathcal{M}^!$ of real analytic functions $f: \HH \rightarrow \C$ which are  modular of weights $(r,s)$ for $r,s  \in \Z$, i.e.,  
\begin{equation}
\label{introfgammaz}
f(\gamma z)  = (cz+d)^r (c \overline{z}+d)^s f(z) \qquad \hbox{ for all } \quad \gamma = \begin{pmatrix} a & b \\ c & d \end{pmatrix} \in \Gamma\ , z \in \HH
\end{equation}
and which furthermore admit an expansion of the form 
\begin{equation} 
\label{fexp} 
 f= \sum_{|k| \leq M}\LL^k  \, \Big(   \sum_{m,n\geq - N}  a^{(k)}_{m,n} q^m \overline{q}^n \Big)   \qquad \hbox{where}  \quad a^{(k)}_{m,n} \in \C \ ,
\end{equation}
for $M, N \in \Z$, 
where $q= \exp(2\pi i z)$ and $\Lef = \log |q| =- 2\pi \mathrm{Im}(z)$. 
The space $\mathcal{M}^!$ is equipped with differential operators $\partial, \overline{\partial}$, closely related to Maass' raising and lowering operators, and a Laplacian $\Delta$.
In \cite{ZagFest}, we defined a subspace $\MI^! \subset \mathcal{M}^!$  of \emph{modular iterated integrals}, generated  from weakly holomorphic modular forms by repeatedly taking primitives with respect to $\partial$ and $\overline{\partial}$. In this installment, we  describe the subspace $\MI^!_1 \subset \MI^!$ of modular iterated integrals of \emph{length  one}. These  correspond  to a  modular incarnation  of the abelian quotient of the relative completion of the fundamental group   \cite{HaGPS, MMV}  of the moduli stack of elliptic curves $\mathcal{M}_{1,1}$. They span the first level in an infinite tower of non-abelian or `mixed' modular functions whose general definition was given in \cite{MMV}, \S18.5. In  \cite{EqEis} we worked out the Eisenstein quotient of this construction; here we spell out  the length one part of the general case.  
\vspace{0.05in}

Examples of   functions  in the class $\MI_1^!$ are given by  real analytic Eisenstein series, which are well-known. 
 Let $r,s\geq 0$ such that  $w=r+s \geq 2$ is even, and define
$$\mathcal{E}_{r,s} = \frac{w!}{(2\pi i )^{w+2}} \frac{1}{2} \sum_{(m,n) \in \Z^2 \backslash \{0,0\}} \frac{\Lef}{(mz+n)^{r+1}(m \overline{z}+n)^{s+1}}\ .$$ These functions are real analytic and modular of weights $(r,s)$, and admit an expansion of the form (\ref{fexp}) (with $N=0$). Following the presentation given in \cite{ZagFest} \S4,
they are the unique solutions to  the following system of differential equations:
$$\partial \mathcal{E}_{r,s} = (r+1) \mathcal{E}_{r+1,s-1} \quad \hbox{ for } s\geq 1 $$
$$\overline{\partial} \mathcal{E}_{r,s} = (s+1) \mathcal{E}_{r-1,s+1} \quad \hbox{ for } r\geq 1 $$
and 
$$\partial \mathcal{E}_{w,0} =  \Lef \GE_{w+2} \qquad \hbox{ , } \qquad  \overline{\partial} \mathcal{E}_{0,w} =  \Lef  \overline{\GE_{w+2}} $$
where $\GE_{2k}$ are the Hecke-normalised holomorphic Eisenstein series:
\begin{equation} \label{GEdef} \GE_{2k}=  -\frac{b_{2k}}{4k}  + \sum_{n\geq 1} \sigma_{2k-1}(n) q^n\ , \qquad k\geq 1 \ . 
\end{equation} 
 Since $\partial \Lef^{-1} \mathcal{E}_{w,0} = \GE_{w+2}$,  the functions  $\Lef^{-1} \mathcal{E}_{w,0}$ are modular primitives (with respect to $\partial$)  of holomorphic Eisenstein series, and are annihilated by the Laplacian. 
 
In this paper, we shall construct \emph{real analytic cusp forms} $\mathcal{H}(f)_{r,s}$ which are canonically associated to any Hecke cusp form, and satisfy an analogous system of differential equations.  It is clear from their construction that they are `motivic', in that their coefficients only involve the periods of pure motives  associated to cusp forms \cite{Scholl}. 
The functions  $\mathcal{H}(f)_{r,s}$ generate $\MI^!_1$, and furthermore, they generate the subspace of $\HM^! \subset \mathcal{M}^!$ of eigenfunctions  of the Laplacian. In other words, the overlap between the space  $\M^!$ and the set of Maass waveforms is exactly described by the functions studied  in this paper.

\subsection{Real Frobenius} The essential ingredient in our construction is the  real Frobenius, also known as complex conjugation.  For all $n\in \Z$ let  $M_{n}^!$ denote the space of weakly holomorphic modular forms of weight $n$.
They admit a Fourier expansion 
\begin{equation} \label{introgdef} f =  \sum_{m \geq -N}  a_m (f) q^m \qquad \hbox{where }  \quad a_m(f) \in \C \end{equation}
for $N\in \Z$. 
Although the differential operator $D= q \frac{d}{dq}$ does not preserve modularity, a  well-known result due to Bol implies  that  its powers define linear maps 
$$D^{n+1} : M_{-n}^! \To M^!_{n+2}\ $$
for all $n\geq 0$.  The quotient $M^!_{n+2}/ D^{n+1} M^!_{-n}$ can be interpreted as a space of modular forms of the second kind \cite{SchollKazalicki}, \cite{BH}. Indeed, it is canonically isomorphic to the algebraic de Rham cohomology of the moduli stack of elliptic curves with certain coefficients, and in particular, admits an action by Hecke operators.  
 Furthermore, one shows \cite{Guerzhoy} that every element  $ f\in M^!_{n+2}/ D^{n+1} M^!_{-n}$ has a \emph{unique representative}  $f\in M^!_{n+2}$ such that 
$$\mathrm{ord}_{\infty} f  \leq \dim S_{n+2}\ .$$
This provides a splitting    $M^!_{n+2}  =  D^{n+1} M^!_{-n} \oplus M^!_{n+2}/ D^{n+1} M^!_{-n}$  which is possibly unnatural, but   is canonical.  We shall use this splitting  to provide canonical constructions and uniqueness statements in the theorems below. A purist may prefer to avoid using this splitting  at the expense of working modulo $ D^{n+1} M^!_{-n}$. 

The `\emph{single-valued involution}' is    a canonical Hecke-equivariant map 
  $$\s : M^!_{n+2}  / D^{n+1} M_{-n}^!  \,  \overset{\sim}{\To} \, M^!_{n+2}  / D^{n+1} M_{-n}^!  \ .$$
It exists in much greater generality  \cite{NotesMot} \S4.1 and is induced, via the comparison isomorphism, by complex conjugation on  Betti cohomology.
By the previous remarks, it lifts to an involution on $M^!_{n+2}$. In fact, it  can be written down explicity on each cuspidal  Hecke eigenspace in terms  of a period matrix
\begin{equation}
\mathrm{P}_f = \begin{pmatrix} 
\eta_f^+  & \omega_f^+ \\
i\eta_f^- &  i \omega_f^{-} \\
\end{pmatrix} \quad \in \quad \mathrm{GL}_2(\C)
\end{equation} 
where $\omega^{+}_f, i \omega_f^-$ are the periods and    $\eta^+_f,  i \eta^-_f$ are the quasi-periods  with respect to a basis $f, f'$ of a cuspidal  Hecke eigenspace.  More precisely, we show that
\begin{equation} \label{introsf} \s(f)   = 
  \Big(\frac{  \eta^{+}_f \omega^{-}_f+  \eta^{-}_f \omega^{+}_f   }{\eta^-_f  \omega^+_f - \eta^+_f \omega^-_f }\Big)    \,    f   +  \Big(\frac{2\, \omega^+_f \omega^-_f }{{\eta^+_f\omega^-_f  - \eta^-_f\omega^+_f  } } \Big)\,   f'  \ .   
\end{equation} 
It does not depend on the choice of basis $f,f'$.
The quantity $\omega^+_f \omega^-_f$ is related to the Petersson norm of $f$, the determinant of the period matrix $\det (\mathrm{P}_f)$
is proportional to a power of $2\pi i$.  The quantity $i (  \eta^{+}_f \omega^{-}_f +\eta^{-}_f \omega^{+}_f )$ is   the \emph{permanent} of the period matrix:
$$\s(f) = - \frac{\mathrm{perm}( \mathrm{P}_f) }{\det( \mathrm{P}_f) } f  + \frac{2 i \omega^+_f \omega^-_f}{\det (\mathrm{P}_f) } f'$$ 

\subsection{Summary of results} 

\begin{thm}  Let $n\geq 0$. Let $f$ be a cuspidal Hecke eigenform of weight $n+2$  for $\SL_2(\Z)$. 
There exists a unique family of real analytic modular functions 
$$\mathcal{H}(f)_{r,s} \quad \in \quad \mathcal{M}_{r,s}^!$$
for all $r+s=n$  and $r,s\geq 0$, 
 satisfying the system of differential equations
\begin{eqnarray} 
\partial \mathcal{H}(f)_{r,s}  &=  & (r+1) \,  \mathcal{H}(f)_{r+1,s-1} \qquad \hbox{ if  } s\geq 1 \nonumber \\
\overline{\partial} \mathcal{H}(f)_{r,s}  &=  & (s+1) \,  \mathcal{H}(f)_{r-1,s+1} \qquad \hbox{ if  } s\geq 1  \nonumber
 \end{eqnarray} 
and 
$$\partial \mathcal{H}(f)_{n,0}  =   \Lef f  \qquad , \qquad  \overline{\partial} \mathcal{H}(f)_{0,n}  =  \Lef\, \overline{\s(f)} \ .$$
 The $\mathcal{H}(f)_{r,s}$ are eigenfunctions of the Laplace operator with eigenvalue $-n$. 
 Equivalently, the $\Lef^{-1}\mathcal{H}(f)_{r,s}$ are harmonic: $\Delta \Lef^{-1}\mathcal{H}(f)_{r,s}=0$.
 \end{thm}
The theorem holds also for weak cusp forms, defining a  canonical map 
$$\mathcal{H}_{r,s}: S^!_{n+2} / D^{n+1} M^!_{-n} \To \mathcal{M}^!_{r,s}$$
for all $r+s=n$, with $r, s\geq 0$.  Since $\s (\GE_{n+2} ) = -\GE_{n+2}$, the
 real analytic Eisenstein series  satisfy identical equations except with a difference of sign (for $\overline{\partial} \mathcal{E}_{n,0}$).   This justifies calling the $\mathcal{H}(f)_{r,s}$  \emph{real analytic cusp forms}.
 
 The theorem can be rephrased as follows.  Consider the real analytic vector-valued function $\mathcal{H}(f) : \HH \rightarrow \C[X,Y]$ defined by  $\mathcal{H}(f) = \sum_{r+s=n} \mathcal{H}(f)_{r,s}(X-z Y)^r(X - \overline{z} Y)^s\ .$  It is equivariant for the standard right action of $\Gamma$ on $\C[X,Y]$ and satisfies
 $$ d \mathcal{H}(f)   =  \pi i \,   f(z) (X-z Y)^n dz +  \pi i \,  \overline{\s(f)} (X-\overline{z}Y)^n d \overline{z}\ .$$

The functions $\mathcal{H}(f)_{r,s}$ are given by the following explicit formula. First, for any weakly holomorphic modular form 
(\ref{introgdef}),  write for all $k\geq 0$ 
\begin{equation} \label{gkdefn} 
f^{(k)} = \sum_{n\in \Z\backslash 0 }   \frac{a_n(f)}{(2n)^k} q^n \ . 
\end{equation}  
It is an iterated primitive for $q \frac{d}{dq}$. For all $r,s \geq 0$ with $r+s =n$ define 
  \begin{equation} \label{Rrsdef} R_{r,s}(f) = (-1)^r \binom{n}{r} \sum_{k=s}^{n} \binom{r}{k-s} (-1)^k  \frac{k!}{\Lef^{k}} f^{(k+1)}\ .
  \end{equation} 
\begin{thm} The functions $\mathcal{H}(f)$ have the following  form:
$$\mathcal{H}(f)_{r,s} =  {a_0(f) \,  \over n+1} \Lef +  \alpha_f (-1)^r \binom{n}{r}  \Lef^{-n}  + R_{r,s}(f) + \overline{R_{s,r}(\s(f))}$$
for some uniquely determined  $\alpha_f \in \C$. 
\end{thm} 
The constant term  $\alpha_f$ can be computed (\S\ref{sectalpha})   from the Fourier coefficients of $f$ and $\s(f)$ in the case when $f$ is  cuspidal, and is given by an odd zeta value in the case when $f$ is an Eisenstein series. 
This dichotomy is due to the fact that the Tate twists of the Tate motive have non-trivial  extensions, but the Tate twists of  the motive of a cusp form do not (in the relevant range). 
When $f$ is holomorphic, the constant $\alpha_f$ is proportional to the  Petersson norm  of $f$. 

When  $f$ is a Hecke cuspidal eigenform with coefficients in a number field $K_f$, then the coefficients in the expansion of $\mathcal{H}(f)_{r,s}$ lie in  a $K_f$-vector space of dimension at most 3 which is spanned by  periods.
 We show furthermore:
 \begin{enumerate}
 \item   If $f$ is a Hecke eigenfunction with eigenvalues $\lambda_m$, then the functions $\mathcal{H}(f)_{r,s}$ satisfy an inhomogeneous Hecke eigenvalue equation with  eigenvalues $m^{-1} \lambda_m$.  See \S\ref{sectHeckeaction} for  precise statements.
 
\item  The action of  $\mathrm{Gal}(\overline{\Q}/\Q)$ on Hecke eigenfunctions extends to an action on the functions $\mathcal{H}(f)_{r,s}$, for every $r,s$. In fact, this action extends to an action of a `motivic'
 Galois group  which  acts on the coefficients   in the expansion (\ref{fexp}) via a  conjectural Galois theory of periods. 
  \end{enumerate}
 The only main ingredient in this  paper is the single-valued involution $\s$, which is derived from the real Frobenius. It would be interesting to replace it with a $p$-adic crystalline Frobenius and  define  $p$-adic versions of real analytic cusp forms (see \cite{CandeloriCastella}).

 \subsection{Weak harmonic lifts and mock modular forms of integral weight}
 Consider the special case $r=n, s=0$.  For the sole  purposes of this introduction  set
 $$ \widetilde{f} =  \Lef^{-1} \mathcal{H}(f)_{n,0}\ .$$
 
 \begin{cor} For every (weakly holomorphic) cusp form  $f$ of weight $n+2$, the function $\widetilde{f}$ is a canonical  weak harmonic lift of $f$. More precisely: 
 $$  \partial  \widetilde{f}  =  f  \qquad \hbox{ and } \qquad     \Delta  \widetilde{f}  =0\ . $$
  In particular $\widetilde{f}$ is a weak Maass waveform. 
  It  is given explicitly by 
 $$\widetilde{f} =  \frac{\alpha_f}{\Lef^{n+1}} + \sum_{k=0}^n \binom{n}{k} (-1)^k \frac{k!}{\Lef^{k+1}} f^{(k+1)} + \frac{n!}{\Lef^{n+1}} \overline{\s(f)^{(n+1)}}$$
 \end{cor} 
 The problem of  constructing weak harmonic lifts has a long history, but an explicit construction has remained elusive.  The existence of weak harmonic lifts in  a much more general setting was proved  in \cite{BF}. Having established existence, the general shape of the Fourier expansion is  easily deduced - the only issue is to determine the unknown Fourier coefficients.  On the other hand a direct, but highly transcendental, construction using Poincar\'e series was given in \cite{BringOno,  Ono}, involving complicated special functions.
This procedure is potentially ill-defined: when the space of cusp forms has dimension greater than one, it involves choices, since there are relations between Poincar\'e series. 
  The question of whether  weak harmonic lifts  have   irrational coefficients or not has been raised   \cite{BOR, Ono}.   Our results imply that these functions, despite appearances, are in fact of geometric, and indeed, motivic, origin.

 The `mock' modular form associated to $\widetilde{f}$ is the complex conjugate of   $\Lef^{n+1}$ times the antiholomorphic part of $\widetilde{f}$. It is harmonic, and given by:
 $$M_f = \alpha_f + n! \, \s(f)^{(n+1)}$$
When $f$ is a Hecke eigenform, $\s(f)$ is given by (\ref{introsf}), which leads to a very simple and explicit construction of mock modular forms of integral weights for $\SL_2(\Z)$.  In the literature, it is customary to rescale the mock modular forms by the Petersson inner product. This gives
\begin{equation} \label{introMfprime} M '_f     =    \alpha'_f    +  (n-1)! \, \sum_{m\in \Z \backslash \{0\}}   \frac{a'_m + \rho\, a_m}{m^{n-1}} q^m \ , 
\end{equation}
where $a_m, a'_m$ are the Fourier coefficients of $f, f'$ respectively, and
$$\rho = -\frac{1}{2} \Big( \frac{\eta^+}{\omega^+} +  \frac{\eta^-}{\omega^-}\Big) \ .$$
The quantity  $\alpha'_f$ is in the field of definition of the $a_m, a'_m$. 

In \S\ref{sectExample},  we compute  this explicitly in the case of Ramanujan's  $\Delta$ function. Let
  \begin{eqnarray}
\Delta   &= &  q  -24 \, q^2 +252\, q^3-   1472\, q^4  + 4830 q^5 + \ldots  \nonumber \\
  \Delta'         &   = & q^{-1} +47709536\, q^2+39862705122\, q^3 + \ldots   \nonumber 
  \end{eqnarray} 
  where $\Delta' \in M^!_{12}$ is the unique  normalised  weakly holomorphic modular form  which has a pole of order  $1$ at the cusp, and whose Fourier coefficients $a_0$, $a_1$ vanish. 
 In this case $a_n, a'_n \in \Z$,  and $a_n$ is the Ramanujan $\tau$-function. The functions   $\Delta, \Delta'$ are a basis for the  de Rham  realisation of  the motive \cite{Scholl} of $\Delta$.   
      If $\rho$ is  irrational (as expected) then the $n$th Fourier coefficient 
  of $M'_f$ is irrational if and only if $a_n \neq 0$.

 Since the space of cusp forms of weight $12$ is one-dimensional, the method of Poincar\'e  series \cite{BringOno} also yields in this case an explicit expression for this mock modular form in terms of special functions. Comparing the Fourier coefficients of the  two gives:

\begin{cor}  \label{corKloost} For all $n> 0$, 
$$  2 \pi \,   n^{\frac{11}{2}}\, \sum_{c=1}^{\infty} \frac{K(-1,n,c)}{c}  I_{11} \Big( \frac{4 \pi \sqrt{n}}{c}\Big) =   a'_n  + \rho    a_n$$ \
where $K$ denotes a  Kloosterman sum  and $I$ a  Bessel function \cite{Ono}. 
\end{cor}

   Since modular forms of level one do not have complex multiplication, Grothendieck's period conjecture, applied to the motives of cusp forms, would imply that its Fourier coefficients are transcendental. The reader will easily be able to generalise the results of this paper to the case of a general congruence subgroup using the results of \cite{SchollKazalicki}\footnote{After we had written this paper,  K. Ono and N. Diamantis kindly pointed out the recent work of Candelori \cite{Candelori}, which is closely related to our construction and applies for modular forms of level $\geq 5$. His formula (6.6) for the Fourier coefficients in the case $n\neq 0$ is very similar to (\ref{introMfprime}).}.   The theory of motives predicts that a modular form has complex multiplication if and only if 
 $\rho$ is an algebraic multiple of the Petersson norm of $f$ by a power of $2 \pi i$.
 This could  explain the phenomena studied  in the recent paper  \cite{BOR} which observed algebraicity of the Fourier coefficients of Maass waveforms  associated to modular forms with complex multiplication.

   \subsection{Contents}
 In \S\ref{sectWeakHol} we  review  the theory of weakly holomorphic modular forms. Much of this material is standard, but many aspects are not widely known and may be of independent interest. In \S\ref{sectMI!} we review some properties of the space $\mathcal{M}^!$ of real-analytic modular forms from \cite{ZagFest}, and its subspaces $\HM^!$ (\S\ref{sectHM}) of Laplace eigenfunctions and $\MI^!$ (\S\ref{sectMI1}) of modular integrals.
 In  \S\ref{sectHecke} we describe the action of Hecke operators on $\HM^!$. Much of this material is well-known. 
 In \S\ref{sectExistence} we prove the existence of weak modular lifts, and  in \S\ref{sectExample} we discuss  Ramanujan's function $\Delta$. 
 \\
 
 \emph{Acknowledgements}.  The author is partially supported by ERC grant GALOP 724638. Many thanks to Larry Rolen for informing me of the problem of finding weak harmonic lifts  during the conference `modular forms are everywhere' in honour of Zagier's 65th birthday, and to Luca Candelori for comments and corrections.

\section{Background on weakly holomorphic modular forms} \label{sectWeakHol}

\subsection{Weakly holomorphic modular forms} The vector space $M^{!}_n$ of weakly holomorphic modular forms of weight $n\in\Z$ is the vector space 
of holomorphic functions $f: \HH \rightarrow \C$ with possible poles at the cusp, which are modular of weight $n$. They  admit a Fourier expansion of the form 
\begin{equation} \label{fweakmodexp} f = \sum_{n\geq - N}  a_n q^n  \qquad \hbox{ where } a_n \in \C \ . 
\end{equation} 
 The space $S^!_n\subset M^!_n$ of cusp forms are those with   $a_0=0$.   The subspace of functions with Fourier coefficients $a_n$ in a ring $R\subset \C$ will be denoted by $M_n^!(R)$. 

Consider the following operator, which does not in general preserve modularity:
\begin{equation}
D = q \frac{d}{dq}\ .
\end{equation}
An identity due to Bol  \cite{Bol} (see also lemma \ref{lemBol}  below) implies, however, that 
$$D^{n+1} : M_{-n}^! \To M_{n+2}^!\ .$$
Its image is contained in the space of cusp forms $S_{n+2}^!$.
Elements in the cokernel of this map can be viewed as modular forms `of the second kind', and can be interpreted as algebraic de Rham cohomology. 
Surprisingly this fact is not well known. It appeared  for the first time implicitly in the work of Coleman \cite{Coleman} on $p$-adic modular forms, and later in \cite{Scholl} and \cite{SchollKazalicki}. A direct proof in the case of level one was given in \cite{BH}.

\begin{thm} \label{M11algdR}  Let $\mathcal{M}_{1,1}$ denote the moduli stack of elliptic curves over $\Q$, and $\V$ the algebraic vector bundle defined by the de Rham  cohomology $H^1_{dR}(\mathcal{E}/\mathcal{M}_{1,1})$ of the universal elliptic curve $\mathcal{E}$ over $\mathcal{M}_{1,1}$, equipped with the Gauss-Manin connection.  Set $\V_n = \mathrm{Sym}^n \,\V$.   For all $n\in \Z$, there is a canonical isomorphism of $\Q$ vector spaces
\begin{eqnarray} \label{MmodDtoHdR}
M_{n+2}^! (\Q)/ D^{n+1} M_{-n}^!  (\Q)  \,  \overset{\sim}{\To}\,  H^1_{dR}( \mathcal{M}_{1,1}; \V_{n})\ .
\end{eqnarray} 
The right-hand side vanishes if $n\leq 0$ or $n$ is odd. 
\end{thm} 
This theorem has a number of consequences that we shall spell out below. Many of these have been known for some time, others apparently not.

There is a canonical decomposition into Eisenstein series and cusp forms
$$ H^1_{dR}( \mathcal{M}_{1,1}; \V_{n})  \ = \  H^1_{\mathrm{cusp}, dR }( \mathcal{M}_{1,1}; \V_{n}) \oplus  H^1_{ \mathrm{eis}, dR}( \mathcal{M}_{1,1}; \V_{n})\ .$$
Via the isomorphism (\ref{MmodDtoHdR}), the latter is generated by Eisenstein series (\ref{GEdef})
$$H^1_{\mathrm{eis},dR}( \mathcal{M}_{1,1}; \V_{n}) =  \Q  \,\GE_{n+2}$$
for all $n\geq 2$, 
and the former is isomorphic to the space of cusp forms
$$H^1_{\mathrm{cusp}, dR}( \mathcal{M}_{1,1}; \V_{n})  =    S_{n+2}^! (\Q)/ D^{n+1} M_{-n}^!  (\Q)\ .$$
Serre duality induces a  pairing on the latter space. Explicitly, if $f , g\in S_{n+2}^!$ are weakly holomorphic cusp forms of weight $n+2$ with Fourier coefficients $a_k(f), a_k(g)$ respectively, it is given   by  \cite{Guerzhoy}, \cite{BH} \S5
\begin{equation}\label{IPdef} \{f, g\} =  \sum_{k\in \Z}  \frac{ a_k(f) a_{-k}(g)}{k^{n+1}}\ .
\end{equation}
It vanishes if  $f$ or $g$ is in the image of the Bol operator $D^{n+1}$. We have
$$\dim_{\Q}  H^1_{\mathrm{cusp}, dR}( \mathcal{M}_{1,1}; \V_{n}) = 2 \, \dim_{\C} S_{n+2}\ .$$
One shows \cite{Guerzhoy}  that every equivalence class
$$[f] \in      M_{n+2}^! / D^{n+1} M_{-n}^!  $$
 has a unique representative $f\in    M_{n+2}^!$ such that
 $$\mathrm{ord}_{\infty} f   \geq - \dim S_{n+2}\ .$$
 Thus we have  a canonical isomorphism
 $$M^!_{n+2} =  D^{n+1} M^!_{-n} \oplus  H^1_{dR}( \mathcal{M}_{1,1}; \V_{n})\ . $$

\subsubsection{Hecke operators}   The isomorphism   (\ref{MmodDtoHdR}) is equivariant with respect to the action of Hecke operators $T_m$, for $m\geq 1$,  which act via the  formula (\ref{usualTmf}) (which we shall re-derive, in a more general context, in \S\ref{sectHecke}).   
 If a formal power series (\ref{fweakmodexp}) has a pole of order $p$ at the cusp, then $T_m f$ has a pole of order $mp$ at the cusp.

The Hecke operators commute with the Bol operator:
$$[ T_m, D^{n} ]= 0 \quad \hbox{ for all } n \  ,$$
which implies that there is an action  of the Hecke algebra for all $n$
$$T_m :  M_{n+2}^! / D^{n+1} M_{-n}^!  \quad \To  \quad M_{n+2}^! / D^{n+1} M_{-n}^!\ . $$
The action of Hecke operators respects the decomposition into Eisenstein series and cusp forms.
In particular, the Eisenstein series $\GE_{2k}$  are normalised  Hecke eigenfunctions: for all $n\geq 2$ and $m\geq 1$, 
\begin{equation} \label{TmonGE} T_m \GE_{n+2} =  \sigma_{n+1}(m) \GE_{n+2}
\end{equation}
The pairing (\ref{IPdef}) is orthogonal with respect to the action of $T_m$ \cite{Guerzhoy}
$$\{T_m f, g\} = \{f, T_m g\} \qquad \hbox{ for all }  f, g \in S_{n+2}^! \ .$$
The space of cusp forms decomposes over $\overline{\Q}$  into Hecke eigenspaces 
$$  H^1_{\mathrm{cusp}, dR}( \mathcal{M}_{1,1}; \V_{n})  \otimes_{\Q} \overline{\Q}  =  \bigoplus_{\llam} H^{dR}_{\llam} \otimes_{K_{\llam}} \overline{\Q}$$
 where $\llam=(\lambda_m)_{m \geq1}$ and $H^{dR}_{\llam}$ is a   2-dimensional $K_{\llam}$ vector space, where $K_{\llam} \subset \R$ is the 
 number field generated by the $\lambda_m$. It is generated by a normalised Hecke eigenform
 $$f_{\llam}   \quad \in \quad M_{n+2}  ( K_{\llam})$$
 which satisfies $T_m f_{\llam}= \lambda_m f_{\llam}$ for all $m$, and a \emph{weak Hecke eigenform} 
 $$f'_{\llam}   \quad \in \quad M^!_{n+2}  ( K_{\llam})$$
 which satisfies for all $m\geq 1$:
 \begin{equation} \label{Tmfprime} T_m f'_{\llam}= \lambda_m \,f'_{\llam}     \pmod{D^{n+1} M^!_{-n} (K_{\llam})}\ .
 \end{equation}
We can assume as a consequence of \cite{BH}, proposition 5.6,  that $f_{\llam}, f'_{\llam}$  satisfy:
$$ \{f'_{\llam} , f_{\llam} \} = 1\ , $$
and furthermore,  that $f'_{\llam}$ has poles at the cusp of order at most $\dim S_{n+2}$.  With these conventions, $H^{dR}_{\llam}$ has a basis
\begin{equation} \label{Hbasis} H^{dR}_{\llam}  =   f_{\llam} K_{\llam}  \oplus f'_{\llam} K_{\llam} 
\end{equation} 
which is well-defined up to transformations  $f_{\llam}' \mapsto f_{\llam}' + a  f_{\llam}$, for $a \in K_{\llam}$. 
\begin{rem}  \label{remcanbasis} One could  fix a `canonical' basis of $H^{dR}_{\llam}$ either by assuming  that the Fourier coefficient $a_1$ of $f'$ is equal to $1$, or by demanding that $ \{f'_{\llam} , f'_{\llam} \} =   0$ (note that  $\{f_{\llam} , f_{\llam} \} =0$  holds automatically). This will not  be required in this paper. The latter condition holds for the basis chosen in 
 \S\ref{sectExample}. 
\end{rem}

\subsubsection{Group cohomology and cocycles}
Let $\Gamma = \SL_2(\Z)$. Let $\VV_n $ denote the local system $ \mathrm{Sym}^n R^1 \pi_* \Q$ on $\mathcal{M}_{1,1}(\C)$ where $\pi: \mathcal{E} \rightarrow \mathcal{M}_{1,1}$ is the universal elliptic curve and $\Q$ is the constant sheaf on $\mathcal{E}(\C)$. Its fiber at the tangent vector $\partial /\partial q$ on the $q$-disk is the vector space
$$V_n = \bigoplus_{i+j=n} \Q X^i Y^j  $$
of homogeneous polyomials in  variables $X, Y$,  corresponding to the standard homology basis of the fiber of the universal elliptic curve.   It admits a right-action by $\Gamma$
$$(X,Y) \big|_{\gamma} = (a X+ b  Y , c X + dY)$$
for $\gamma $ of the form (\ref{introfgammaz}). Recall that the space of cocycles 
$Z^1 (\Gamma ; V_n)$
is the  $\Q$-vector space generated by functions 
$\gamma \mapsto C_{\gamma} : \Gamma \rightarrow V_n$
satisfying the cocycle equation 
$$C_{gh} = C_g\big|_h   + C_h  \quad \hbox{ for all }  \quad g, h \in \Gamma \ .$$ 
Such a cocycle is uniquely determined by $C_S, C_T$, where
\begin{equation}
S = \begin{pmatrix}  0  &  -1 \\ 1 & 0 \end{pmatrix} \qquad  , \qquad 
T = \begin{pmatrix}  1  &  1 \\ 0 & 1 \end{pmatrix}\ .
\end{equation} 
The  polynomials $C_S, C_T$ satisfy a system of equations called the cocycle equations.
  A cocycle is called cuspidal if $C_T=0$. The subspace of coboundaries 
$B^1(\Gamma; V_n)$ is the $\Q$-vector space generated by cocycles of the form 
$$C_{\gamma} = P\big|_{\gamma} - P$$
for some $P \in V_n$. The cohomology group  is defined by 
$$H^1(\Gamma; V_n)   = Z^1(\Gamma; V_n ) / B^1(\Gamma; V_n)\ .$$
There is a natural action of Hecke operators on $H^1(\Gamma; V_n)$. In fact, this action naturally lifts to an action on the space of
cocycles $Z^1(\Gamma; V_n)$ which preserves $B^1(\Gamma;V_n)$ \cite{Ma0}.

Complex conjugation on $\mathcal{M}_{1,1}(\C)$ induces an involution called the real Frobenius $F_{\infty}$ upon $H^1(\Gamma; V_n)$ (and in fact $Z^1(\Gamma; V_n)$). 
It acts on $\Gamma$ by conjugation by 
$$\epsilon = \begin{pmatrix} 1 & 0 \\ 0 & -1 \end{pmatrix}$$
and on $V_n$ by right-action by $\epsilon$, i.e., $(X,Y) \mapsto (X,-Y)$ (see \cite{MMV} \S5.4). 
In particular, there is a canonical decomposition 
\begin{equation} \label{H1pmdecomp} H^1(\Gamma; V_n) =   H^1(\Gamma; V_n)^+ \oplus H^1(\Gamma; V_n)^-
\end{equation} 
into $F_{\infty}$-eigenspaces. The first is spanned by classes of cocycles $C$ such that $C_S$ is $\epsilon$-invariant (even), the second by cocycles which are anti-invariant (odd). 

Finally, there is an inner product on $H^1_{\mathrm{cusp}}(\Gamma; V_n)$ induced by a pairing between cocycles and compactly-supported cocycles  \cite{MMV}, \S8.3 which gives:
$$\{  \ , \  \} : Z^1(\Gamma; V_n) \times Z^1_{\mathrm{cusp}} (\Gamma; V_n) \To \Q \ , $$
a formula for which was given by Haberlund,  e.g. \cite{MMV} (2.11).

\subsubsection{Eichler-Shimura isomorphism}
 The following corollary is a  consequence of  a mild extension of Grothendieck's algebraic de Rham theorem.

\begin{cor} There is a canonical isomorphism 
\begin{equation} \label{compisom} \mathrm{comp}_{B,dR} :  H^1_{dR}( \mathcal{M}_{1,1}; \V_{n}) \otimes_{\Q} \C  \overset{\sim}{\To}  H^1(\Gamma; V_n) \otimes_{\Q} \C \ . \end{equation}
It respects the action of Hecke operators on both sides. 
\end{cor} 
In particular, the comparison  isomorphism respects the decomposition into Eisenstein and cuspidal parts. 
It can be computed  as follows.  Fix a point $\tau_0 \in \HH$. 

\begin{defn} For every  $f\in M_{n+2}^!$, where $n\geq 0$,   let us write   
\begin{equation} \label{Ffdefn} F_f(\tau  ) =  \int^{\tau_0}_{\tau}   (2 \pi i)^{n+1}   f(\tau)(X- \tau Y)^{n} d\tau \ .
\end{equation}
The integral converges since $\tau_0$ is finite. It defines a $1$-cocycle
\begin{equation} \label{Cfdef} C^f_{\gamma}(X,Y)= 
F_f(\gamma \tau) \big|_{\gamma} - F_f(\tau)  \qquad  \in  \quad Z^1(\Gamma; V_{n} \otimes \C)\ .
\end{equation}
which is independent of $\tau$. 
\end{defn} 
Changing $\tau_0$ modifies this cocycle by a coboundary. We
deduce   a linear map
\begin{eqnarray} M_{n+2}^!  &\To &  H^1(\Gamma; V_{n}\otimes \C)  \nonumber   \\
f & \mapsto & [C^f] \nonumber \end{eqnarray} 
which is well-defined, i.e., independent of the choice of point $\tau_0$, and Hecke-equivariant.  One easily shows  (see \cite{ESforMock} or  version 1 of \cite{ZagFest}) that  $f \in D^{n+1} M_{-n}^!$ if and only if $[C^f] \in B^1(\Gamma; V_n)\otimes \C$, and hence the previous map  descends to an isomorphism  
$$ M_{n+2}^!/ D^{n+1} M_{-n}^!  \overset{\sim}{\To}  H^1(\Gamma; V_{n}) \otimes \C$$
which corresponds via  theorem \ref{M11algdR}  to   the comparison isomorphism $\mathrm{comp}_{B,dR}$.

\subsubsection{Period matrix}
Since the comparison isomorphism is Hecke-equivariant, it respects the decomposition into Hecke eigenspaces.

Let $H_{\llam}^B$ denote the  Hecke eigenspace of $H^1(\Gamma; V_n)$ corresponding to the eigenvalues $\llam$. It is a $K_{\llam}$-vector space of dimension $2$ and admits a decomposition 
$$H_{\llam}^B  = H_{\llam}^{B,+}  \oplus H_{\llam}^{B,-}$$
into $\pm$ eigenspaces with respect to the real Frobenius $F_{\infty}$.

The comparison isomorphism induces a canonical  isomorphism
$$ \mathrm{comp}_{B,dR} \quad :  \quad H^{dR}_{\llam} \otimes_{\Q} \C  \overset{\sim}{\To} H^B_{\llam} \otimes_{\Q} \C\ .$$

\begin{defn} Let us choose  generators $P_{\lambda}^{\pm}$ of  $H^{B, \pm}_{\llam}$ respectively, and a basis (\ref{Hbasis}) for  $H^{dR}_{\llam}$. A \emph{period matrix} is the comparison isomorphism $\mathrm{comp}_{B,dR}$ written with respect to these  bases:
\begin{equation} \label{periodmatrix} 
\mathrm{P}_{\llam}= \begin{pmatrix} 
\eta_{\llam}^+  & \omega_{\llam}^+ \\
i\eta_{\llam}^- &  i \omega_{\llam}^- \\
\end{pmatrix}
\end{equation} 
It is well-defined up to  multiplication  on the left by a  diagonal matrix with entries in $K_{\llam}^{\times}$, which   reflects the ambiguity in the choices of $P_{\llam}^{\pm}$ up to scalar, and multiplication on the right by a lower triangular matrix with $1$'s on the diagonal. \end{defn}

From the compatibility of the period isomorphism with complex conjugation and real Frobenius (\S\ref{sectRealFrobandsv}), the numbers 
$ \omega_{\llam}^{\pm} , \eta_{\llam}^{\pm}$
are real.  The  $\omega_{{\llam}}$ are the usual periods of $f_{\llam}$, the numbers  $\eta_f$ could be called its `quasi-periods' and seem not to have been considered in the literature. 
It was proved in \cite{BH} theorem 1.7 that 
$$\det(\mathrm{P}_{\llam})  \quad  \in \quad (2\pi i )^{n+1} K^{\times}_{\llam}\ .$$

\subsubsection{Hodge theory}  The de Rham cohomology group $H^1_{dR}(\mathcal{M}_{1,1}; \V_n)$ admits an increasing weight filtration $W$ and a decreasing Hodge filtration $F$ by $\Q$-vector spaces.  The basis $(\ref{Hbasis})$ is compatible with the Hodge filtration. 

Similarly, $H^1(\Gamma; V_n)$ is equipped with an increasing filtration $W$ compatible with the weight filtration on de Rham cohomology via the comparison isomorphism. 

Thus $H^1(\Gamma; V_n)$ defines a mixed Hodge structure, and is in fact the Betti realisation of a motive \cite{Scholl}. The latter admits a decomposition  (as motives)
$$H^1 ( \mathcal{M}_{1,1} ; \V_n) =    H^1_{\mathrm{cusp}} ( \mathcal{M}_{1,1} ; \V_n)\oplus  H^1_{\mathrm{eis}} ( \mathcal{M}_{1,1} ; \V_n)$$
where 
$ H^1_{\mathrm{eis}, dR} ( \mathcal{M}_{1,1} ; \V_n)  \cong \Q(-n-1)$
and $H^1_{\mathrm{cusp}} ( \mathcal{M}_{1,1} ; \V_n)$ decomposes, over $\overline{\Q}$, as a direct sum of motives $V_{\llam}$ of rank $2$ of type $(n+1, 0)$ and $(0, n+1)$.

\subsubsection{Real Frobenius and single-valued map}   \label{sectRealFrobandsv}
All the constructions in this paper are simply a consequence of  complex conjugation.  The comparison isomorphism fits in the  following  commuting diagram
$$
\begin{array}{ccc}
\mathrm{comp}_{B,dR} \ : \  H^1_{dR}(\mathcal{M}_{1,1}; \V_n) \otimes_{\Q} \C  & \overset{\sim}{\To}    &  H^1(\Gamma; V_n)  \otimes_{\Q} \C \\
 \downarrow &   & \downarrow   \\
 \mathrm{comp}_{B,dR} \ : \   H^1_{dR}(\mathcal{M}_{1,1}; \V_n) \otimes_{\Q} \C &  \overset{\sim}{\To}   &    H^1(\Gamma; V_n)  \otimes_{\Q} \C
\end{array}
$$
where the vertical map on the left is the $\C$-antilinear isomorphism  $c_{dR}$ which is the   identity on $H^1_{dR}(\mathcal{M}_{1,1}; \V_n)$ and   complex conjugation on $\C$; and the vertical map on the right is $F_{\infty} \otimes c_B$ where $c_B$ is complex conjugation on the coefficients. 

It follows that the real Frobenius $F_{\infty}$ induces an isomorphism which we have had occasion to call the `single-valued map' \cite{NotesMot}, \S4.1: 
$$ \s :  H^1_{dR}(\mathcal{M}_{1,1}; \V_n) \otimes_{\Q} \C \overset{\sim}{\To}    H^1_{dR}(\mathcal{M}_{1,1}; \V_n) \otimes_{\Q} \C \ . $$
It is none other than the composition 
$$ \s=  \comp_{B, dR}^{-1}  \circ( F_{\infty}  \otimes \id) \circ  \comp_{B, dR}\ .$$ 
It induces an isomorphism on every Hecke eigenspace
$$ \s:  H^{dR}_{\llam}\otimes \C \overset{\sim}{\To} H^{dR}_{\llam}\otimes \C\ .$$
Written in the basis (\ref{Hbasis}), it is given explicitly by the matrix
$$  \overline{P_{\llam}}^{-1}  P_{\llam}   =  {i \over \det \mathrm{P}_f}     \begin{pmatrix} 
 \eta_{\llam}^{+}\omega_{\llam}^{-}  +\omega^{+}_{\llam} \eta^{-}_{\llam}       & 2  \omega_{\llam}^{+} \omega_{\llam}^{-} \\
  - 2  \eta_{\llam}^{+} \eta_{\llam}^{-}  &     - \eta_{\llam}^{+}\omega_{\llam}^{-}  - \omega_{\llam}^{+} \eta_{\llam}^{-}     \\
\end{pmatrix}
$$
On the Hecke eigenspace corresponding to Eisenstein series, which is a pure Tate motive $\Q(-n-1)$, $\s$ is  multiplication by  $-1$ and  $\s(\GE_{2n+2}) = - \GE_{2n+2}$.  For cusp forms, 
$$\s(f)   = 
  \Big(\frac{  \eta^{+}_f \omega^{-}_f+  \eta^{-}_f \omega^{+}_f   }{\eta^-_f  \omega^+_f - \eta^+_f \omega^-_f }\Big)    \,    f   +  \Big(\frac{2\, \omega^+_f \omega^-_f }{{\eta^+_f\omega^-_f  - \eta^-_f\omega^+_f  } } \Big)\,   f'  \ .   
$$
From this formula for $\s(f)$  and the equation $\{f, f'\}=1$ we find that 
\begin{equation} \label{sff} \{ \s(f) , f \} =  \frac{ 2 i \omega_{\llam}^+ \omega_{\llam}^- }{ \det(\mathrm{P}_f)} 
\end{equation}
which by proposition 5.6 of \cite{BH} is proportional  (depending on one's choice of normalisation) to  the  Petersson norm of $f$.
One could define the Petersson norm of $f'$ to be  $\{ \s(f') , f' \}$. The  off-diagonal  entries of the single-valued period matrix are proportional to the permanent
$$\mathrm{perm} \, (\mathrm{P}_{\llam})   =  i (\eta_{\llam}^{+}\omega_{\llam}^{-}  +\omega^{+}_{\llam} \eta^{-}_{\llam}  ) \ .$$
\begin{rem}
The  constructions above clearly work for the motives \cite{Scholl} of any cuspidal eigenforms of integral weight for congruence subgroups of $\SL_2(\Z)$. In the case when the motive admits complex multiplication,  the ratios $\omega^+_{\llam}/\omega^-_{\llam}$ and $\eta^+_{\llam} /\eta^-_{\llam}$ will be algebraic.  It follows that in this case, the ratio
$$  \frac{\mathrm{perm} \, (\mathrm{P}_{\llam})}{\det\, (\mathrm{P}_{\llam})} = \frac{\eta^+_{\llam} /\eta^-_{\llam}  +    \omega^+_{\llam}/\omega^-_{\llam}}{  \eta^+_{\llam} /\eta^-_{\llam} -  \omega^+_{\llam}/\omega^-_{\llam}}  $$
will also be algebraic. 
\end{rem}
\section{The space $\M^!$ of  non-holomorphic modular forms} \label{sectMI!}

We recall some definitions from \cite{ZagFest}. 
Let 
\begin{equation} \label{Ldef}
\Lef = \log |q| = i \pi (z -\overline{z}) = - 2\pi y
\end{equation} 
which is modular of weights $(-1,-1)$. 
Recall that $\M^!$  is the complex vector space of real analytic modular functions (\ref{introfgammaz}) admitting an expansion of the form (\ref{fexp}).
 Let  $\M\subset \M^!$  denote the subspace of functions  for which  $N$ is zero, i.e., such that  $a^{(k)}_{m,n}$ vanishes if $m$ or $n$ is negative. If $\M^!_{r,s}$ denotes the subspace of functions of modular weight $(r,s)$, then 
$$\M^! = \bigoplus_{r,s} \M^!_{r,s}$$
is a bigraded algebra.   The \emph{constant part} of $f$ is defined to be 
$$f^0 = \sum_{|k| \leq M}\LL^k  a^{(k)}_{0,0} \quad \in \quad \C[\Lef^{\pm}]\ .$$
We say that $f$ is a cusp form if $f^0=0$. The subspace of cusp forms is denoted $\mathcal{S}^!\subset \mathcal{M}^!$ and its component of weights $(r,s)$ is denoted $\mathcal{S}^!_{r,s}$.

\subsection{Differential operators} \label{sectDiffop} There exist bigraded derivations 
$$\partial,  \overline{\partial} : \mathcal{M}^! \To \mathcal{M}^! $$
of bidegrees $(1,-1)$ and $(-1,1)$, whose restrictions to a component $\mathcal{M}^!_{r,s}$  are
$$\partial_r = (z- \overline{z}) {\partial \over \partial z } + r \quad \hbox{and}  \quad \partial_s = (\overline{z}- z) {\partial \over \partial  \overline{z}} + s\ $$
respectively.  The following lemma is straightforward consequence:
\begin{eqnarray} \label{partialraction} \partial_r \Lef^k q^m \overline{q}^n  &=  & (2 m \Lef + r+ k) \Lef^k q^m \overline{q}^n   \\ 
\overline{\partial}_s \Lef^k q^m \overline{q}^n  &=  & (2 n \Lef + s+ k) \Lef^k q^m \overline{q}^n \ .\nonumber 
\end{eqnarray} 
It is valid for any integers $k,m,n,r,s$.
\begin{lem} \label{lemker}  For all $r,s$,  the kernels of $\partial, \overline{\partial}$ are given by 
\begin{eqnarray} (\mathcal{M}_{r,s}^! \cap \ker \partial_r )  & \cong  & \Lef^{-r} \overline{M}^!_{s-r} \nonumber \\
 (\mathcal{M}_{r,s}^! \cap \ker  \overline{\partial}_s )  & \cong & \Lef^{-s} M^!_{r-s}\ .\nonumber
\end{eqnarray} 
In particular, $(\ker \partial) \cap (\ker \overline{\partial}) = \C[ \Lef^{\pm}]$. 
\end{lem} 
Since there exist weakly holomorphic modular forms of negative weight, it follows that primitives in $\mathcal{M}^!_{r,s}$, unlike the space $\mathcal{M}_{r,s}$,  are never unique.

The  bigraded Laplace operator  is the linear map
$$\Delta : \mathcal{M}^! \To \mathcal{M}^! $$
of bidegree $(0,0)$, which acts on  $\mathcal{M}^!_{r,s}$  by 
\begin{equation} \label{Deltarsdef} \Delta_{r,s} =  - \overline{\partial}_{s-1} \partial_r + r(s-1) = - \partial_{r-1} \overline{\partial}_s + s(r-1) \ . 
\end{equation}
Define linear operators
$$\sfh , \sfw : \mathcal{M}^! \To \mathcal{M}^!$$
 by  $\sfh( f) = (r-s) f$ and $\sfw(f) = (r+s) f$ for all $f\in \mathcal{M}_{r,s}^!$. 
\begin{lem} \label{lemoperatoridentities} These operators satisfy  the equations
$$[ \partial, \overline{\partial} ]  =   \sfh \quad , \quad  [\sfh, \partial] = 2 \partial \quad  , \quad [\sfh, \overline{\partial}] = - 2 \overline{\partial}$$ 
i.e., $\partial, \overline{\partial}$ generate a copy of $\ssl_2$. Furthermore, 
$$  [ \partial, \Lef ] = [\overline{\partial}, \Lef] = [ \partial,  \Delta  ] = [   \overline{\partial}, \Delta] = 0 \ , \hbox{ and}$$
$$[\Lef, \Delta ] = \sfw \, \Lef \quad , \quad [\Lef, \sfw] = 2 \Lef  \quad , \quad [\Lef, \sfh]=[\Delta, \sfw]=0\ .$$ 
\end{lem}
The equations $[\partial, \Lef] = [\overline{\partial}, \Lef]$ imply that $\Lef$ is constant for the  differential operators $\partial, \overline{\partial}$, and justify calling $f^0$ the `constant' part. 

\subsection{Bol's operator} \label{sectBol}
Recall the operator
$$D = q \frac{d}{dq}  = \frac{1}{2\pi i}\frac{\partial}{\partial z}\ .$$
\begin{lem}   \label{lemBol}  For all $n\geq 0$,  the following identity of operators  holds:
\begin{equation} \label{IDforBol}   \LL^{n+1} \Big(  {1 \over  \pi i} {\partial \over \partial z}\Big)^{n+1} =  \partial_0 \partial_{-1} \ldots \partial_{-n} \ .
\end{equation}
\end{lem} 
\begin{proof}
Consider the Weyl ring $\Q[x, {\partial\over \partial x}]$ and write $\theta = x {\partial\over \partial x}$. Then the following identity is easily verified  for all $n\geq 1$:
\begin{equation}\label{thetaWeylid}   \theta (\theta-1) \ldots (\theta-n) = x^{n+1} \big(\textstyle{\partial \over\partial x}\big)^{n+1} \ .
\end{equation} 
For example, it can be tested on $x^m$ for $m\geq 0$.  Set $d_z = (\pi i)^{-1} \partial /\partial z$ and observe that 
$\partial_r = \LL d_z + r$.
Since $d_z \LL=1$, there is an isomorphism   $   \C[x, {\partial /\partial x} ] \overset{\sim}{\rightarrow}  \C[ \LL, d_z]$ sending
$x$ to $\LL$ and ${\partial / \partial x} $ to $d_z$. The image of $\theta+r$ is $\partial_r$, so $(\ref{IDforBol})$ is  equivalent to $(\ref{thetaWeylid})$. 
\end{proof} 
Since $\partial$ commutes with $\Lef$, we can write 
\begin{equation}  \label{EquationforBol}
D^{n+1} \Big|_{\mathcal{M}^!_{-n,\bullet}} = \Big( \frac{\partial}{2 \Lef}\Big)^{n+1}\Big|_{\mathcal{M}^!_{-n,\bullet }}\ . 
\end{equation} 
This defines for all $s \in \Z$ a linear map
$$D^{n+1} :  \mathcal{M}^!_{-n,s} \To \mathcal{M}^!_{n+2,s}\ .$$
Its complex conjugate defines a map $\overline{D}^{n+1}:    \mathcal{M}^!_{r,-n} \rightarrow \mathcal{M}^!_{r,n+2}$ for all $r$. 

\subsection{Vector-valued modular forms} \label{vectvaluedandFrs}
Call a real analytic function $F:\HH \rightarrow V_n \otimes \C$ \emph{equivariant} if  for every $\gamma \in \SL_2(\Z)$ and all $ z\in \HH$ it satisfies:
$$F(\gamma z) \big|_{\gamma} = F(z)\ .$$
There is a correspondence \cite{ZagFest} \S7.2, between sections of  $ V_n \otimes \C$ and families of functions $F_{r,s}: \HH \rightarrow \C$ for $r+s=n$ with $r,s\geq 0$.  It is given by writing
\begin{equation}
F(z) = \sum_{r+s= n}  F_{r,s} (X- zY)^{r} (X- \overline{z} Y)^s\ .
\end{equation}
 Then $F$ is equivariant if and only if each $F_{r,s}$ is modular of weights $(r,s)$. Furthermore, $F$ admits an expansion in $\C[q^{-1}, \overline{q}^{-1}, q, \overline{q} ]] [z, \overline{z}]$ if and only if each $F_{r,s} \in \mathcal{M}_{r,s}^!$.

A special case of \cite{ZagFest}, proposition 7.2  implies that 
\begin{equation}\label{dFgeneral}  d F  =   \pi i \, f(z) (X-z Y)^n dz   +  \pi i  \, \overline{g(z)} (X- \overline{z} Y)^n d\overline{z}
\end{equation}
 if and only if the following system of equations holds:
\begin{eqnarray}  \label{partialFgeneral}
\partial F_{r,s} & = &   (r+1) F_{r+1, s-1} \qquad \hbox{ for all } s \geq 1    \\
\overline{\partial} F_{r,s} & = &   (s+1) F_{r-1, s+1} \qquad \hbox{ for all } r \geq 1 \nonumber   \\
\partial F_{n,0} = \Lef f & \quad \hbox{ ,  } \quad  &  \overline{\partial} F_{0,n} = \Lef \overline{g}\ .\nonumber 
\end{eqnarray}  
In the present paper, we only consider the case  $f,g \in M_{n+2}^!$.
\subsection{Some useful lemmas}
 \begin{lem} \label{lempartialiterate}  Let $f \in \mathcal{M}^!_{r,s}$, and write $h=r-s$. Suppose that $\partial f=0$. Then 
 $$ \partial \overline{\partial}^k f =  k ( h - k+1 )\,  \overline{\partial}^{k-1} f$$
for all  integers $k\geq 0$. 
 \end{lem} 
 
 \begin{proof}  It follows from  $[\partial, \overline{\partial}]=\mathsf{h}$ and induction that 
 \begin{equation} \label{partialandbarequation}
 \partial \overline{\partial}^k - \overline{\partial}^k \partial = \sum_{i+j=k-1, i,j\geq 0 } \overline{\partial}^i \,  \mathsf{h} \, \overline{\partial}^j
 \end{equation} 
 Applying this to $f$ gives the stated formula. 
 \end{proof} 
 
 \begin{cor} \label{corvanishing}
 Let $f \in \mathcal{M}^!_{r,s}$ with $r\geq s$. Let $h =r-s\geq 0$.  Then if 
 $$\partial f = 0 \qquad \hbox{ and } \qquad \overline{\partial}^{h+1} f =0 $$
 then  $f\in \C \Lef^{-r}$ if $r=s$ and $f$ vanishes if $h>0$. 
 \end{cor} 
 \begin{proof}
 By lemma \ref{lemker}, $\partial f = 0 $ implies that $f\in \Lef^{-r} \overline{M}_{s-r}^!$. In particular, the coefficients in its expansion (\ref{fexp})
satisfy $a^{(k)}_{m,n}(f)=0$ if $m\neq 0$. This property is stable under $\overline{\partial}$, so the same holds for  all $\overline{\partial}^{n} f$.  
Again by  lemma \ref{lemker}, $ \overline{\partial}^{h+1} f =0$ implies that $\overline{\partial}^{h} f\in \Lef^{-s} M^!$, and its coefficients satisfy 
$a^{(k)}_{m,n}(\overline{\partial}^{h} f)=0$ if either  $m\neq 0$ or $n\neq 0$. It follows that $\overline{\partial}^{h} f \in \C[\Lef^{\pm}]$.
If $h=0$, then $ f \in \mathcal{M}^!_{r,r}$ and  we have shown that $f\in \C \Lef^{-r}$. 
Now if $h>0$,   $\overline{\partial}^{h} f  \in \mathcal{M}^!_{s,r} $, and it follows that $\overline{\partial}^h f =0$ since all powers of $\Lef$ lie on the diagonal $\sfh=0$. Applying the previous lemma to $f$ we find that
 $$ \partial \overline{\partial}^k f =  k ( h - k+1 )\,  \overline{\partial}^{k-1} f$$
and so by decreasing induction on $k$, for $k\leq h$, we deduce that $\overline{\partial}^k f$ vanishes for all $k\geq 0$. This  completes the proof. 
 \end{proof}

\section{The space $\HM^!$ of harmonic functions} \label{sectHM}

\begin{defn}
Let $\HM^! \subset \mathcal{M}^!$ (respectively, $\HM \subset \mathcal{M}$) denote the space of functions which are eigenvalues of the Laplacian.
For any $\lambda \in \C$ let 
$$\HM^!(\lambda) =  \ker \Big( \Delta-\lambda  : \mathcal{M}^!\To \mathcal{M}^!\Big)$$
denote the eigenspace with eigenvalue $\lambda$. 
\end{defn} 

\begin{lem} The space  $\HM^!(\lambda)$ is stable under the action of $\ssl_2$:
$$\partial, \overline{\partial} : \HM^!(\lambda) \To \HM^!(\lambda)$$
and furthermore, multiplication by $\Lef$ is an isomorphism
\begin{equation}\label{LmultonHM}  \Lef : \HM^!_{r+1,s+1}(\lambda) \To \HM^!_{r,s}(\lambda-r-s)\ .
\end{equation}
\end{lem}
\begin{proof}
The first equation follows since  $[\nabla, \partial]=[\nabla, \overline{\partial}]=0$ by lemma \ref{lemoperatoridentities}. For the   second, 
$[\Lef, \Delta] = w \Lef $  implies that 
 if $\Delta F = \lambda F$, then $\Delta(\Lef F) = (\lambda - w) \Lef F$. 
\end{proof}

The lemma remains true on replacing $\HM^!(\lambda)$ by $\HM(\lambda) = \HM^!(\lambda) \cap \mathcal{M}$.

\begin{lem} Every Laplace eigenvalue is an integer: 
$$\HM^! = \bigoplus_{n \in \Z}  \HM^!(n)\ .$$
Every element $F\in \HM^!(\lambda)$ has a unique decomposition 
\begin{equation} \label{Fdecompah0} F =  F^h + F^0 + F^a\ , 
\end{equation}
where $F^0 \in \C[\Lef^{\pm}]$ is the constant part of $F$, and 
\begin{eqnarray} F^h  &\in &  \C[q^{-1}, q]][\Lef^{\pm}] \nonumber  \\ 
F^a  &\in &  \C[\overline{q}^{-1}, \overline{q}]][\Lef^{\pm}] \nonumber 
\end{eqnarray}
are the (weakly)  `holomorphic' and `antiholomorphic' parts of $F$, and have no constant terms. 
Furthermore, each piece is an eigenfunction: $\Delta F^{\bullet} = \lambda F^{\bullet}$ for $\bullet \in \{h,0,a\}$. 
\end{lem} 

\begin{proof}
 This was proved for the space $\HM$ in \cite{ZagFest}, lemma 5.2. The proof is more or less identical for $\mathcal{M}^!$.
 \end{proof}

One can be  more precise     (\cite{ZagFest}, \S5.1). Let $F\in \HM^!_{r,s}$ with eigenvalue $\lambda\in \Z$.  Let $w = r+s$ be the total weight. Then there exists a $k_0 \in \Z$ such that 
\begin{equation}\label{F0shape}  F^0  \quad \in \quad  \C \Lef^{k_0}   \oplus \C \Lef^{1-w-k_0}
\end{equation} 
where $k_0 < 1-w- k_0$ and  $\lambda = k_0 (1-w-k_0)$, and furthermore:
\begin{equation} \label{FaFhshape} F^h  \quad \in \quad  \bigoplus_{k=k_0}^{-s}  \C[q^{-1}, q]]  \Lef^{k}  \qquad , \qquad   F^a  \quad \in \quad  \bigoplus_{k=k_0}^{-r}  \C[\overline{q}^{-1}, \overline{q}]]  \Lef^{k} \ .
\end{equation}

\section{The space $\MI_1^!$ of weak modular primitives} \label{sectMI1}

The subspace $\MI^! \subset \mathcal{M}^!$ of modular iterated integrals was defined in \cite{ZagFest}. 

\begin{defn} Let $\MI^!_{-1} = 0$.  For every $k\geq 0$, let  $$\MI_k^! \subset \bigoplus_{r,s\geq 0} \mathcal{M}^!_{r,s}$$  be the largest subspace which is concentrated in the positive quadrant of $\mathcal{M}^!$ (with modular weights $(r,s)$ with $r, s\geq 0$) with the property that
\begin{eqnarray} \label{MIkcond} \partial \MI^!_k   & \subset  & \MI^!_k   +      M^! [\Lef] \otimes \MI^!_{k-1} \\
\overline{\partial} \MI^!_k   &  \subset & \MI^!_k   +      \overline{M^!} [\Lef] \otimes \MI^!_{k-1}  \nonumber 
\end{eqnarray} 
for all $k \geq 0$. 
We define $\MI^! = \sum_k \MI^!_k$.  It is closed under complex conjugation. 
\end{defn} 
We call the increasing filtration  $\MI^!_k \subset \MI^!$  the length. 
In this paper we shall focus only length $\leq 1$. We first dispense with the subspace of length $0$. 

\begin{prop}
$\MI^!_0 =  \C[\Lef^{-1} ]$. 
\end{prop} 

\begin{proof} Firstly, the space $\C[\Lef^{-1}]$ satisfies the conditions of the definition since $[\partial, \Lef] = [\overline{\partial}, \Lef]=0$, and so $\C[\Lef^{-1}] \subset \MI^!_0$.  Now let $F \in \MI^!_0$ be of  modular weights $(n,0)$, where $n\geq 0$. Since $\partial F$ has weights $(n+1,-1)$, which lies outside  the positive quadrant, we must by (\ref{MIkcond}) and $\MI^!_{-1}=0$ have $\partial F =0$. 
 Similarly, the element $F'=\overline{\partial}^n F$ has weights $(0,n)$ and  so  $\overline{\partial} F'=0$ since it  also lies outside the positive quadrant.   By corollary \ref{corvanishing},   $F$ vanishes if $n>0$ and $F\in \C$ if $n=0$.   By complex conjugation, it  follows that $\MI^!_0$ vanishes in modular weights $(0,n)$  and $(n,0)$ for all $n\geq 1$ and is contained in $\C$ in weights $(0,0)$. We can now repeat the argument for any $F\in \MI_0^!$ of modular weights $(n,1)$ by replacing $F $ with $\Lef F$ and arguing as above. We deduce that $\MI_0^!$ vanishes in all weights $(n,1)$ and $(1,n)$ for $n\geq 2$ and is contained in $\C \Lef^{-1}$ in weights $(1,1)$.  Continuing in this manner, we conclude that $\MI_0^! \subset \C[\Lef^{-1}]$. 
\end{proof}

\subsection{Modular iterated integrals of length one}It follows from the previous proposition that $\MI^!_1$ is the largest subspace of $\mathcal{M}^!$ which satisfies
\begin{eqnarray} \label{MIkcondlength1} \partial \MI^!_1   & \subset  & \MI^!_1   +      M^! [\Lef^{\pm}]  \\
\overline{\partial} \MI^!_1   &  \subset & \MI^!_1   +     \overline{M^!} [\Lef^{\pm}]  \nonumber \ .
\end{eqnarray} 
In particular, any element $F \in \MI^!_1$ of weights $(n,0)$, with $n\geq 0$,  satisfies 
$$\partial F = \Lef f$$
for some $f\in M_{n+2}^!$ weakly holomorphic of weight $n+2$. We call such an element  a  \emph{modular primitive} of $ \Lef f$.   It is necessarily a Laplace eigenfunction with eigenvalue $-n$ since $(\Delta + n) F = - \overline{\partial} \partial F= 0 $ by (\ref{Deltarsdef}). 

\begin{rem} As a consequence, $\Lef^{-1} F$ satisfies $\partial \Lef^{-1} F = f$ and $\Delta \Lef^{-1} F =0$. It  is therefore what is known as  a weak harmonic lift of $f$. 
\end{rem}

 \begin{prop}  \label{propweaklift} Let $n\geq 0$.  Let $f $ be a weakly holomorphic  modular form of weight $n+2$, and let $X_{n,0} \in \mathcal{M}^!$ be a primitive of $\Lef f$:
 $$\partial X_{n,0} = \Lef f\ .$$
  Then there exist unique elements $X_{r,s} \in \mathcal{M}_{r,s}^!$, for $r, s \geq 0$ and $r+s=n$ such that 
  \begin{eqnarray} \label{Xrseqns} \partial X_{r,s} & = &  (r+1) X_{r+1,s-1}  \quad \hbox{ for } s \geq 1   \\
   \overline{\partial} X_{r,s} & = &  (s+1) X_{r-1,s+1}   \quad \hbox{ for } r \geq 1   \nonumber
   \end{eqnarray} 
   and 
$$\overline{\partial} X_{0,n} = \Lef \overline{g}$$
for some $g\in M^!_{n+2}$ a weakly  holomorphic modular form of weight $n+2$.   It follows that $(\Delta+n) X_{r,s} =0$ for all $r+s=n$, i.e., $X_{r,s} \in \HM^!(-n)$. 
\end{prop} 

\begin{proof}
Suppose that $X_{n,0}$ is a primitive of $\Lef f$. Define $X_{r,s} $ by the formula
\begin{equation} \label{Xrsdef} X_{r,s} =  \frac{\overline{\partial}^s}{s!} X_{n,0}  
\end{equation}
for all $r+s = n$, $r, s\geq0$.  The second equation of $(\ref{Xrseqns})$ holds for all $r,s$.
For the first equation, apply identity $(\ref{partialandbarequation})$ to $X_{n,0}$ to obtain
$$\partial \overline{\partial}^k X_{n,0}  - \overline{\partial}^k \partial X_{n,0}  =   k(n-k+1)  \overline{\partial}^{k-1} X_{n,0}\ .$$
 For $k\geq 1$ the second term is $\Lef \,\overline{\partial}^k f$, which vanishes. Therefore, by (\ref{Xrsdef}),
 $$ k! \, \partial X_{n-k,k} = k(n-k+1) (k-1)! X_{n-k+1,k-1}$$
 which is exactly the first equation of $(\ref{Xrseqns})$. Applying $\mathsf{h} = [\partial, \overline{\partial}] $ to $X_{0,n}$, and using the equations $(\ref{Xrseqns})$, one finds that $\partial \overline{\partial} X_{0,n}=0$. Therefore
 $$\overline{\partial} X_{0,n} \ \in \,  \mathcal{M}^!_{-1,n+1} \cap \ker \partial_{-1} ,$$
 and by lemma \ref{lemker},  it follows that   $ \overline{\partial} X_{0,n} = \Lef  \overline{g}$ for some $g \in M^!_{n+2}$ as claimed. 
 Finally, the fact that the $X_{r,s}$ are Laplace eigenfunctions with eigenvalue $-n$ follows easily from (\ref{Deltarsdef}),  (\ref{Xrseqns}) and, when $n=0$,  the equations $\partial X_{n,0}= \Lef f$, $\overline{\partial} X_{0,n} = \Lef \overline{g}$. 
 \end{proof}

\begin{rem} \label{remXassection} If we define $\XX: \HH \rightarrow \V_n \otimes \C$ by  
$$ \XX  = \sum_{r+s=n}  X_{r,s} (X- zY)^r (X- \overline{z}Y)^s$$
then $\XX$ is modular equivariant, and equations $(\ref{Xrseqns})$ are equivalent to
$$ d \XX =  {2 \pi i \over 2} \Big(    f(z) (X- z Y)^{n} dz  + \overline{g(z)} (X- \overline{z} Y)^n d\overline{z} \Big) \ . $$
The fact that the coefficients $X_{r,s}$ are eigenfunctions is equivalent to the identity 
$$ {\partial^2 \over \partial z \partial \overline{z} }\XX = 0 \ .$$ 
\end{rem}

We now turn to uniqueness. 

 \begin{lem} Let $X_{n,0}$ (respectively $X'_{n,0}$) be  modular  primitives of $\Lef f$, and let 
 $X_{r,s}, g$  (resp.  $X'_{r,s}, g'$)   be the quantities defined in proposition \ref{propweaklift}. Then  there exists an $\xi \in M^!_{-n}$ such that  for all $r+s= n$ and $r,s\geq0$ 
$$X'_{r,s} - X_{r,s} = \LL^{-n}\frac{\overline{\partial}^s} {s!} \overline{\xi}\ , $$ 
 and 
 $$ g'  - g  =  \frac{1}{n!}\overline{\partial}^{n+1} \Lef^{-n-1}  \xi \ .$$
In other words, $g$ and $g'$ are equivalent modulo $D^{n+1} M^!_{-n}$. 
\end{lem}
\begin{proof} By lemma \ref{lemker}, $X'_{n,0}- X_{n,0} \in \Lef^{-n} \overline{M^!}_{-n}$. Apply (\ref{Xrsdef}) and (\ref{EquationforBol}) to conclude. 
\end{proof} 

\begin{cor}
If $X_{n,0}$ is a primitive of $\Lef f$, and $X_{r,s}$, $g$ are as defined in proposition \ref{propweaklift},   then $Y_{r,s} = \overline{X}_{s,r}$ is a system of solutions to the equations  $(\ref{Xrseqns})$ and satisfies 
$$\partial Y_{n,0} = \Lef g \quad \hbox{ and } \quad   \overline{\partial} Y_{0,n} = \Lef \overline{f} \ . $$
 Therefore complex conjugation reverses the roles of $f$ and $g$. 
\end{cor}

\subsection{Harmonic functions and structure of $\MI^!_1$}
We show that the modular primitives of  proposition \ref{propweaklift} generate $\MI^!_1$ under multiplication by $\Lef^{-1}$.   This section can be skipped and is not required for the rest of the paper.

\begin{prop} \label{propMIinHM}  Modular integrals of length one lie in the harmonic subspace of $\mathcal{M}^!$:
$$\MI^!_1  \subset \HM^!\ .$$
More precisely, any element $F \in \MI^!_1$  of modular weights $(r,s)$ can be uniquely decomposed as a linear combination of elements
$$F = \sum_{0\leq k \leq  \min\{r,s\}}    F_k$$
where $F_k \in \MI^!_1$ also has modular  weights $(r,s)$ and satisfies: 
$$\Delta  F_{k} =  (k-1)(r+s-k)\, F_k \ .$$
Specifically, if $r\geq s$, each $F_k$ is of the form $F_k = \Lef^{-k}\overline{\partial}^{s-k} X_k$, for some $X_k$ a  modular primitive of $\Lef f_k$, where $f_k \in M_{r+s+2-2k}^!$ is weakly holomorphic.   

In the case  $s\leq r$, then we can take $F_k = \Lef^{-k} \partial^{r-k} \overline{X}_k$, with $X_k$ a modular primitive  of $\Lef g_k$, where $g_k \in M_{r+s+2-2k}^!$ is weakly holomorphic.   
\end{prop} 
\begin{proof}
Suppose that  $F$ is in $\MI^!_1$ of modular weights $(r,s)$ with $r\geq s$.  We show by induction on $s$ that  it is a linear combination:
\begin{equation}\label{Findsum} F = \sum_{0\leq k \leq s} F_k  \qquad \hbox{ where } \qquad F_k = \Lef^{-k}\frac{\overline{\partial}^{s-k}}{(s-k)!} X_k  \quad \in \quad \mathcal{M}_{r,s}^!
\end{equation} 
where $\partial X_k \in \Lef M^!$, and hence $X_k$ is a modular primitive of total weight $r+s-2k$. By proposition \ref{propweaklift}, $X_k$ is a Laplace eigenfunction with eigenvalue $2k-r-s$, and it follows from $[\Delta, \overline{\partial}]=0$ and (\ref{LmultonHM}) that  $ F_{k}$ is also an eigenfunction with eigenvalue 
$$(2k-w) + (w-2k) + (w-2k-2) + \ldots +(w-2) = (1-k)(w-k)$$
where we write $w=r+s$.  Since these eigenvalues are distinct for distinct values of $0\leq k \leq w/2$,  the $F_k$ are linearly independent and the decomposition is unique.

The statement (\ref{Findsum}) is true for $F$ of modular weights $(n,0)$:  in that case 
(\ref{MIkcondlength1}), together with the fact that $\partial F$ lies outside the positive quadrant, implies that
$$\partial F \quad \in \quad  M^![\Lef^{\pm}]$$
and hence $\partial F = \Lef f$, for some $f \in M^!_{n+2}$.  Therefore $F$ is a modular primitive of $\Lef f$, and by proposition \ref{propweaklift}, an eigenfunction of the Laplacian with eigenvalue $-n$.  Now suppose that $F \in \MI^!_1$ of modular weights $(r,s)$ with $r\geq s \geq 0 $ and suppose that (\ref{Findsum}) is true for all smaller values of $s$. Then since 
$$\partial F \quad  \in \quad \MI^!_1  +   M^![\Lef^{\pm}]$$
has modular weights $(r+1,s-1)$, the induction hypothesis implies that
$$\partial F =  \Lef^{1-s} f + \sum_{0\leq k \leq  s-1} \Lef^{-k}\frac{\overline{\partial}^{s-1-k}}{(s-1-k)!} X_k $$
for some $f \in M^!_{r-s+2}$. From the proof of proposition \ref{propweaklift}, each term $\frac{\overline{\partial}^{s-1-k}}{(s-1-k)!} X_k$ has a modular primitive $\frac{\overline{\partial}^{s-k}}{(s-k)!} X_k$. Define $X_s$ via the formula
$$  \Lef^{-s} X_s =  F-    \sum_{0\leq k \leq  s-1} \Lef^{-k}\frac{\overline{\partial}^{s-k}}{(s-k)!} X_k\ .$$
Then $X_s$ is a modular primitive of $\Lef f$ and $F$ is of the required form, completing  the induction step.  The case where $s\geq r$ follows by complex conjugating, which reverses the roles of $r$ and $s$.  Taking both cases together implies the first statement. \end{proof}
In particular :

\begin{itemize}
\item an element $F\in \MI^!_1$ of modular weights $(n,0)$ is necessarily an eigenfunction of the Laplacian with eigenvalue $-n$. 
\item an element $F\in \MI^!_1$ of modular weights $(n-1,1)$ is a linear combination of two eigenfunctions of the Laplacian with possible eigenvalues 
$\{-n, 0\}$. 

\item an element $F\in \MI^!_1$ of total weight $w$ can have eigenvalues in the set 
$$\{-w \ , \  0 \  , \  w-2 \ , \  2(w-3) \ ,  \ 3(w-4) \ ,  \ \ldots \ ,  \ \textstyle{w \over 2}(1- {w\over 2}) \}$$
\end{itemize} 
\begin{rem} Elements in $\MI^!_k$ for $k\geq 2$ are no longer harmonic and satisfy a more complicated structure with respect to the Laplace operator.  See, e.g., \cite{EqEis} \S11.3-4.
\end{rem}

\subsection{Ansatz for primitives} Recall that  for  $f\in M^!_{n+2}$ a weakly holomorphic modular form, 
the quantities 
$f^{(k)}$ and $R_{r,s}(f)$ were defined in (\ref{gkdefn}) and (\ref{Rrsdef}).
In particular, 
\begin{eqnarray} \label{Rn0} R_{n,0}(f)  &= & (-1)^n \sum_{k=0}^n \binom{n}{k} (-1)^k \frac{k!}{\Lef^k} f^{(k+1)}   \\
R_{0,n}(f)  &= & (-1)^n \frac{n!}{\Lef^n} f^{(n+1)} \nonumber  
\end{eqnarray} 
We shall write $R_{r,s}$ instead of $R_{r,s}(f)$ when $f$ is understood.

\begin{prop}  \label{propRrsproperties} The functions $R_{r,s}$ satisfy 
\begin{eqnarray} \partial_r R_{r,s} & = &   (r+1) R_{r+1, s-1} \quad \hbox{ for all }\, s\geq 1 \nonumber \\
 \partial_n R_{n,0} & = & (-1)^n \Lef f^{(0)} \ .\nonumber 
 \end{eqnarray}
Furthermore 
$$\overline{\partial}_s R_{r,s}  = \begin{cases}  (s+1) R_{r-1,s+1}  \quad \hbox{ if } r \geq 1 \nonumber \\ 0 \qquad \qquad  \qquad \qquad \hbox{ if } r=0 
\end{cases} \ . 
$$
\end{prop}
\begin{proof} Let us write 
$$S_{r,s} = \sum_{k=s}^{n} \binom{r}{k-s} (-1)^k \frac{k!}{\Lef^{k}}  f^{(k+1)}\ ,$$
and show that for all $s\geq 1$, 
$$\partial_r S_{r,s} + s S_{r+1, s-1} =0 \ . $$
 We first verify using (\ref{partialraction}) that 
\begin{eqnarray} \partial_r \Lef^{-k} f^{(k+1)}  &= &\Big( \sum_{m\in \Z\backslash 0} {2m \over (2m)^{k+1}} a_m \Lef^{1-k } q^m \Big)+ (r-k)\, \Lef^{-k} f^{(k+1)} \nonumber \\
& = & \Lef^{1-k} f^{(k)}  + (r-k) \Lef^{-k} f^{(k+1)}  \ . \nonumber 
\end{eqnarray} 
It follows that 
\begin{multline} \partial_r S_{r,s} +s S_{r+1,s-1} = \sum_{k\leq n} \binom{r}{k-s} (-1)^k k! \Big( \Lef^{1-k} f^{(k)} + (r-k) \Lef^{-k} f^{(k+1)}\Big)    \nonumber \\ 
+ s \Big( \binom{r+1}{k-s+1} (-1)^k k! \Lef^{-k} f^{(k+1)} \Big) 
\end{multline}
Using $r+s=n$, the right-hand side reduces to 
$$\sum_{k\leq n} \binom{r}{k-s} (-1)^k  \frac{k!}{\Lef^{k}} f^{(k+1)} \Big[ \frac{- (k+1)(n-k)}{k-s+1} + (r-k) + \frac{s(r+1)}{k-s+1}\Big] = 0$$
since the term in square brackets simplifies to zero.  Finally, since
$$R_{r,s} = (-1)^r \binom{n}{r} S_{r,s}\ ,$$
we find that for all $s\geq 1$, 
\begin{eqnarray}  \partial_r R_{r,s} - (r+1) R_{r+1, s-1}  &=  & (-1)^r \frac{(r+s)!}{r! s!} \partial_r S_{r,s} - (-1)^{r+1} \frac{(r+s)!(r+1)}{(r+1)! (s-1)!}  S_{r+1,s-1}
\nonumber \\
&= & (-1)^r  \frac{(r+s)!}{r! s!} \big( \partial_r S_{r,s} + s S_{r+1,s-1}\big) \nonumber
\end{eqnarray}
which vanishes.  This proves the first equation.  For the second, by (\ref{Rn0}), we have
$$\partial_n R_{n,0} = (-1)^n \sum_{k=0}^n \frac{n!}{(n-k)!} (-1)^k \big[ \Lef^{1-k} f^{(k)} + (n-k) \Lef^{-k} f^{(k+1)}\big]$$
By telescoping, only the first term in square brackets (for $k=0$), and the second term (for $k=n$) survive. The latter is zero, and the former is 
  exactly $(-1)^n\Lef f^{(0)}$.

  For the last part, compare $\overline{\partial}_s R_{r,s}$ and $(s+1) R_{r-1,s+1}$ using:
     $$(-1)^r \binom{n}{r} \binom{r}{k-s} (s-k) = (-1)^{r-1} \binom{n}{r-1} \binom{r-1}{k-s-1} (n-r+1)$$
  where $n=r+s$. The case $r=0$ is immediate from lemma \ref{lemker}.
\end{proof}

\begin{lem}  \label{lemEsection} Let $E: \HH \rightarrow V_n\otimes \C$ be real analytic and $T$-equivariant such that
$$\frac{\partial E}{\partial z}= 0 \quad , \quad \frac{\partial E}{\partial \overline{z}} = c (X- \overline{z} Y)^n$$
where $c\in \C$. Then $c=0$ and $E = \frac{\alpha}{(\pi i)^n} Y^n$ for some $\alpha \in \C$. Writing 
$$E= \sum_{r+s=n} E_{r,s} (X- zY)^r (X- \overline{z} Y)^s$$
we find that 
$$E_{r,s} = \alpha  (-1)^r \binom{n}{r}  \Lef^{-n} \ .$$
If $E$ is modular equivariant and $n>0$ then  $\alpha$ vanishes.
\end{lem} 
\begin{proof}
Consider the function $e: \HH \rightarrow \C$ obtained by composing $E$ with $V_n \otimes\C  \rightarrow V_n \otimes \C /Y \C\cong \C$.   It is the coefficient of $X^n$ in $E$. 
It satisfies 
$\frac{\partial e}{\partial z}= 0 $ and $ \frac{\partial e}{\partial \overline{z}}= c$ and therefore $e= c \overline{z} + \beta$ for some $\beta \in \C$. Since $T$ fixes $Y$ and acts on $X$ by $T(X)= X+Y$, the condition of $T$-invariance implies that $e(z+1)=e(z)$. This forces $c=0$.  
It follows that $\frac{\partial}{\partial z} E= \frac{\partial}{\partial \overline{z}} E=0$ and so $E$ is constant. By $T$-invariance, $E$ lies in $V_n^T = \C Y^n$, and hence $E = \frac{\alpha}{(\pi i)^n} Y^n$ for some $\alpha \in \C$.  But
\begin{multline} E =\frac{\alpha}{ (\pi i)^n} Y^n =   \frac{\alpha}{ (\pi i)^n}\frac{1}{(z-\overline{z})^n} \Big(  (X- \overline z Y) - (X- zY) \Big)^n \nonumber \\
  = \alpha    \Lef^{-n} \sum_{r+s=n}  (-1)^r \binom{n}{r}  (X- z Y)^r (X- \overline{z} Y)^s \end{multline} 
since $\Lef = \pi i (z - \overline{z})$, which proves the formula for $E_{r,s}$. 

Finally, if $E$ is  modular equivariant,  $E_{r,s}  \in \C \Lef^{-n}$ is modular of weights $(r,s)$ with $r+s=n>0$. But $\Lef^{-n}$ is modular of weights $(n,n)$, which implies that $E_{r,s}= 0$. 
\end{proof} 

\begin{cor} Let $f\in M_{n+2}^!$ be a weakly holomorphic modular form. Let $X_{n,0}$ be a modular primitive of $\Lef f$, and let $X_{r,s}$ and $g \in M_{n+2}^!$ be as determined by proposition \ref{propweaklift}. Then the zeroth Fourier coefficients of  $f$ and $g$ are conjugate:
$$a= a_0 (f) = \overline{a_0(g)} $$
and  there exists some $\alpha \in \C$ such that 
\begin{equation} \label{Xrsexplicitformula} X_{r,s} =  {a \,  \over n+1} \Lef +  \alpha (-1)^r \binom{n}{r}  \Lef^{-n}  + R_{r,s}(f) + \overline{R_{s,r}(g)}
\end{equation}
for all $r,s\geq 0$ and $r+s=n$. 
\end{cor} 

\begin{proof} Let $a=a_0(f)$.   Define 
$$Y_{r,s}  =   {a \,  \over n+1} \Lef   + R_{r,s}(f) + \overline{R_{s,r}(g)}\ . $$
We first check that the expression for $Y_{r,s}$ satisfies the equations (\ref{Xrseqns}).   By (\ref{partialraction}), we have $\partial_r \Lef = (r+1) \Lef$, and by proposition \ref{propRrsproperties}, we deduce that 
$\partial_r Y_{r,s} = (r+1) Y_{r+1,s-1}$ for all $s\geq 1$.  
Similarly, using the fact that $n$ is even, we check that
$$\partial_n Y_{n,0} = a \, \Lef +      \partial_n R_{n,0}(f) =  a \, \Lef + (-1)^n f^{(0)} \, \Lef = \Lef \, f\  \ . $$ 
By complex conjugating,  $\overline{\partial}_s Y_{r,s} = (s+1) Y_{r-1,s+1}$ for all $r\geq 1$, and 
$$  \overline{\partial}_n Y_{0,n} = a \, \Lef + \overline{\partial}_n \overline{R_{n,0}(g)} = \Lef (a + \overline{g}^{(0)})\ . $$
Define $E_{r,s}= X_{r,s}- Y_{r,s}$. The function $E= \sum_{r+s=n} E_{r,s} (X- z Y)^r (X- \overline{z} Y)^s$ satisfies 
$$\frac{\partial E}{\partial z} =  0 \quad \hbox{ and } \quad \frac{\partial E}{\partial \overline{z}} = \pi i ( \overline{a_0(g)} - a ) (X - \overline{z} Y)^n $$
by (\ref{dFgeneral}).
It is a real analytic and $T$-invariant section of $V_n \otimes \C$, since  $X_{r,s}$ and $Y_{r,s}$ are $T$-invariant. By the previous lemma we conclude that there exists an $\alpha \in \C$ such that 
$$X_{r,s} = \alpha (-1)^r \binom{n}{r}  \Lef^{-n}  + Y_{r,s}$$ 
for all $r+s =n$, and furthermore, that $\overline{a_0(g)} = a$.   
\end{proof}
We shall determine the unknown coefficient $\alpha$ using Hecke operators.  Another way to prove the corollary is to use the fact that $X_{r,s}$ are eigenfunctions of the Laplacian (proposition  \ref{propMIinHM}) and the explicit shape (\ref{F0shape}) and (\ref{FaFhshape}) for the latter.  We chose the approach above since it explains the origin of the indeterminate  coefficient $\alpha$, and since functions in $\MI$ are not harmonic in general. 

\begin{cor}
A modular primitive of $\Lef f$, if it exists, is of the form:
\begin{equation}
X_{n,0} = \frac{a}{n+1}\Lef + \frac{\alpha}{\Lef^{n}} + \frac{n!}{\Lef^{n}}\overline{g^{(n+1)}}+  \sum_{k=0}^n \binom{n}{k} (-1)^k \frac{k!}{\Lef^k} f^{(k+1)} \ .
\end{equation}
\end{cor} 

\subsection{Example: real analytic Eisenstein series}
Let $\mathcal{E}_{r,s}$ denote the functions defined in the introduction. 
By \cite{ZagFest}, proposition 4.3, and  equation (\ref{Xrsexplicitformula}), we have
$$\mathcal{E}_{r,s} = \mathcal{E}^0_{r,s}  + R_{r,s}(\GE_{w+2}) + R_{s,r}(\overline{\GE_{w+2}})\ ,$$
where
$$ \mathcal{E}^0_{r,s}=  - \frac{B_{w+2}}{2 (w+1)(w+2)} \Lef + (-1)^r \binom{w}{r} \frac{w!}{2^{w+1}} \zeta(w+1) \Lef^{-w}\ .$$
In this example the coefficient $\alpha$ is  an odd zeta value, which is the period of a non-trivial extension of Tate motives,  and is conjecturally transcendental. It can be obtained as a special value of a suitably-defined $L$-function of $\mathcal{E}_{r,s}$, which will be discussed elsewhere. 

\section{Hecke operators} \label{sectHecke} 
We review some  basic properties of Hecke operators. For any $\alpha \in \mathrm{GL}_{2}(\R)$ write 
\begin{equation}\label{alphanotation}   \alpha = \begin{pmatrix}  a_{\alpha} & b_{\alpha} \\ c_{\alpha} & d_{\alpha} \end{pmatrix}\ ,
\end{equation}
i.e., $a,b,c,d$ are the standard generators on the affine ring $\Or(\mathrm{GL}_2)$. 
\subsection{Definition} 
Let $f: \HH \rightarrow V_n \otimes \C$ be real analytic and equivariant. Let $m\geq 1$ be an integer, and let 
$M_m$ denote the set of $2\times 2$ matrices with integer entries which have  determinant $m$.  The Hecke operator is defined by the formula\footnote{The reason for the factor $m^{-1}$  is that $f$  is a function; the usual formula for Hecke operators involves one-forms: for  $\alpha$ as in (\ref{alphanotation}),
$$ d (\alpha z) = { \det(\alpha) \over (c_{\alpha} z+d_{\alpha})^2} dz\ , $$
and the $\det(\alpha)$ accounts for an extra multiple of $m$ in the  formula for $T_m$.}

$$T_m f (z) = {1 \over m} \sum_{\alpha \in \Gamma\backslash M_m} f(\alpha z) \big|_{\alpha}\ .$$
Since $f$ is equivariant, it follows that  for all $\gamma \in \Gamma$, 
$$f(\gamma \alpha z)\big|_{\gamma \alpha} = f(\gamma ( \alpha z) )\big|_{\gamma} \big|_{\alpha} = f(\alpha z)\big|_{\alpha}$$
and hence the formula for $T_m f$ is well-defined. The set of cosets $ \Gamma \backslash M_m $ is finite, and are described below.
  Since right multiplication by any $\gamma \in \Gamma$ defines a bijection of cosets 
$ \Gamma \backslash M_m \overset{\sim}{\To}  \Gamma \backslash M_m $
we deduce from the  calculation 
$$(T_m f)(\gamma z)\big|_{\gamma} = \sum_{\alpha \in  \Gamma \backslash M_m} f( \alpha \gamma z)\big|_{\alpha \gamma} =\sum_{\alpha' \in  \Gamma \backslash M_m} f( \alpha'  z)\big|_{\alpha'} = T_mf(z)$$
 that  $T_m f : \HH \rightarrow V_n \otimes \C$ is equivariant. Via the dictionary \S\ref{vectvaluedandFrs} between equivariant vector-valued modular forms and modular forms of weights $(r,s)$, we deduce an action of $T_m$ on the latter.  It is given by the following formula. 
 
 \begin{lem} \label{lemTmf} If $f$ is real analytic modular of weights $(r,s)$ then 
 $$T_m f = \sum_{\alpha \in \Gamma \backslash M_m} { m^{r+s-1}  \over (c_{\alpha}z+d_{\alpha})^r (c_{\alpha} \overline{z} +d_{\alpha})^s }  f(\alpha z) $$
 and is  real-analytic modular of weights $(r,s)$. 
  \end{lem} 
  \begin{proof}
 For any  $\alpha$ as in (\ref{alphanotation}), 
 $$(X - \alpha z Y) \big|_{\alpha} = {\det(\alpha)  \over (c_{\alpha} z+d_{\alpha})} (X-zY) \ .$$
 Writing  $f$ in the form 
 $f = \sum_{r+s=n} f_{r,s} (X-z Y)^r (X-\overline{z} Y)^s$,  we find that 
 $$T_m f = {1 \over m} \sum_{\alpha \in \Gamma \backslash M_m} \sum_{r+s=n} f_{r,s} (\alpha z) {\det(\alpha)^n  \over (c_{\alpha}z+d_{\alpha})^r (c_{\alpha} \overline{z} +d_{\alpha})^s } (X-zY)^r (X- \overline{z} Y)^s\ .$$
 Reading off the coefficients gives the stated formula. 
\end{proof}
\subsection{Properties}

\begin{lem}  \label{lemTmprop} View $T_m$, multiplication by $\Lef$, and $\partial, \overline{\partial}, \Delta$ as operators acting  on  real analytic modular functions. Then they  satisfy
$$ m \, T_m \, \Lef  =  \Lef \,  T_m\ ,$$
$$ [T_m ,\partial ] = [T_m, \overline{\partial}]=0\ .$$
The second equation implies that $ [ T_m , \Delta ]  =0$.
\end{lem} 
\begin{proof}
For any $\alpha$  as in (\ref{alphanotation}),
$$\mathrm{Im}(\alpha z) = { \det(\alpha) \over (c_{\alpha}z+d_{\alpha})(c_{\alpha}\overline{z}+d_{\alpha})} \mathrm{Im}(z)\ .$$
If $f$  is modular of weights $(r,s)$, then $\mathrm{Im}(z) f$ is modular of weights $(r-1,s-1)$ and 
\begin{eqnarray} 
T_m (\mathrm{Im}(z)  f)  &=  &  \sum_{\alpha \in \Gamma \backslash M_m}  { m^{r+s-3}  \over (c_{\alpha}z+d_{\alpha})^{r-1} (c_{\alpha}\overline{z}+d_{\alpha})^{s-1}  }\mathrm{Im} (\alpha z) f(\alpha z) \nonumber \\
&  = &  \mathrm{Im}(z)  \sum_{\alpha \in \Gamma \backslash M_m}  { m^{r+s-2}  \over (c_{\alpha}z+d_{\alpha})^{r} (c_{\alpha}\overline{z}+d_{\alpha})^{s} }f(\alpha z) 
   \quad =  \quad   {1 \over m }  \mathrm{Im}(z) T_m f(z) \ .\nonumber
\end{eqnarray}
The first equation follows from  $\Lef = - 2 \pi \mathrm{Im} (z)$. One verifies  for any $\alpha$ of the form (\ref{alphanotation}) (dropping the subscripts $\alpha$ for convenience):
$$\partial_r  \big( (cz+d)^{-r} f(\alpha z)\big) = (cz+d)^{-r-1} (c \overline{z} +d)  \big(\partial_r f \big)(\alpha z)$$
$$\overline{\partial}_s  \big( (c\overline{z}+d)^{-s} f(\alpha z)\big) = (cz+d) (c \overline{z} +d)^{-s-1}  \big(\overline{\partial}_s f \big)(\alpha z)$$
Since  $\partial_r f$ is modular of weights $(r+1, s-1)$ \begin{eqnarray} T_m (\partial_r f)  &=  & \sum_{\alpha \in \Gamma \backslash M_m} {m^{r+s-1}  \over (c_{\alpha}z+d_{\alpha})^{r+1}(c_{\alpha} \overline{z} +d_{\alpha})^{s-1}} (\partial_r f)(\alpha z)  \nonumber \\ 
&= &  \sum_{\alpha \in \Gamma \backslash M_m} {m^{r+s-1}  \over (c_{\alpha} \overline{z} +d_{\alpha})^{s}}  \partial_r  \Big( {f(\alpha z) \over (c_{\alpha}z+d_{\alpha})^{r}} \Big) \quad = \quad  \partial_r T_m (f)  \nonumber
\end{eqnarray}
which proves that $[T_m , \partial ] f =0$. The statement for  $\overline{\partial}$ follows by complex conjugation. The equation $[T_m ,\Delta] f=0$ follows from the definition of the Laplacian \S\ref{sectDiffop}. 
\end{proof} 
By \cite{Serre} \S5.2  Lemma 2,  a complete set of representatives for the set of cosets $\Gamma \backslash M_m$ are given by the $\sigma_1(m) =  \sum_{d|m} d $ integer matrices
$$ \begin{pmatrix} a & b \\ 0 & d \end{pmatrix} \quad \hbox{ where }  \quad ad=m \ , \  a\geq 1 \ , \  0 \leq b < d\ .$$
It follows from lemma \ref{lemTmf} that for any $f$ modular of weights $(r,s)$, we have
\begin{equation} \label{usualTmf} T_m f (z) =  m^{w-1} \sum_{ad=m, a,d>0} {1 \over d^w} \sum_{0\leq b <d} f \Big( { az+ b \over d} \Big)
\end{equation} 
where $w= r+s$ is the total weight of $f$, which is the usual formula. 
It follows that the $T_m$ commute and satisfy the following relations \cite{Serre} \S5.1: 
\begin{eqnarray} T_m T_n  &=  & T_{mn} \qquad \qquad \qquad \qquad  \hbox{ if } (m,n) \hbox{ coprime}  \nonumber \\ 
 T_p T_{p^{n}}   &=  & T_{p^{n+1}} + p^{w-1} T_{p^{n-1}}  \qquad \hbox{ if } p \hbox{ prime}, n\geq 1 \nonumber 
 \end{eqnarray}
viewed as operators acting on modular forms of total weight $w$. 

\subsection{$q$-expansions}
The Hecke operators do not preserve the spaces $\mathcal{M}$ and $\mathcal{M}^!$.  Indeed,  it follows from the definitions that the map $f (z)\mapsto f(\frac{az+b}{d})$ acts via
 $$ \Lef^k q^m \overline{q}^n\quad \mapsto \quad   \Big( \frac{a}{d}\Big)^k e^{ 2 \pi i (m -n) \frac{b}{d}}\,  \Lef^k q^{\frac{ma}{d}}   \overline{q}^{\frac{na}{d}} \ . $$ 
The following corollary is a consequence of  formula (\ref{usualTmf}) and continuity. 
 
\begin{cor}
Let $R \subset \C$. The Hecke operator $T_N$ defines a linear map 
$$T_N : \mathcal{M}^!(R) \To \mathcal{M}^{[N],!}(R[\zeta_N])$$ 
where $\zeta_N$ is a primitive $N$th root of unity, and $\mathcal{M}^{[N],!}(S)$ is the space of real analytic modular forms
which admit an expansion in
 $S [q^{-1/N},\overline{q}^{-1/N},  q^{1/N},\overline{q}^{1/N}]] [\Lef]$.
\end{cor} 
It is well-known that  for $r \in \Z$, 
\begin{equation} \label{sumofexps} \sum_{0 \leq b <d }  e^{ 2\pi i r \frac{b}{d}} = \begin{cases} 0 & \hbox{ if  }  d\!\! \not| \, r \\ d & \hbox{ if } d|r \end{cases}\ . 
\end{equation}
\begin{cor} \label{corfHecke} Let  $f  \in \mathcal{M}^!_{r,s}$  with an expansion
\begin{equation}\label{corfexp} f = \sum  a^{(k)}_{m,n} \Lef^k q^m \overline{q}^n\end{equation}
satisfying the property that for all $d|N$, $d>1$, 
\begin{equation} \label{Heckecond} a^{(k)}_{m,n} = 0 \quad \hbox{ whenever } 0 \not\equiv m \equiv n \pmod{d}\ .
\end{equation}
Then $T_N f \in \mathcal{M}^!_{r,s}$. 
More precisely, one has the formula
\begin{equation} \label{corTNfexp} T_N f = \sum_{k,\mu,\nu} \, \alpha^{(k)}_{\mu, \nu}  \, \Lef^k q^{\mu} \overline{q}^{\nu}
\end{equation}
where
 $$   \alpha^{(k)}_{\mu, \nu} =  \sum_{a |(N, \mu, \nu), a\geq 1}  a^{w-1}\Big(\frac{a^2}{N}\Big)^k a^{(k)}_{\frac{\mu N}{a^2}, \frac{\nu N}{a^2}}  \ .
$$
In particular, if $f \in \mathcal{M}_{r,s}$, and satisfies (\ref{Heckecond}), then $T_N f \in \mathcal{M}_{r,s}$. 
\end{cor} 
\begin{proof}
Apply  $T_N$ to the expansion of $f$ via formula (\ref{usualTmf})  to deduce  that
$$T_N f = \sum_{k,m,n} N^{w-1} \sum_{ad=N, a,d>0} \Big(\frac{a}{d}\Big)^k \frac{1}{d^w} \sum_{0\leq b <d} a^{(k)}_{m,n} e^{ 2\pi i (m-n)\frac{b}{d}} \Lef^k q^{\frac{ma}{d}} \overline{q}^{\frac{na}{d}}\ .$$
With the assumption (\ref{Heckecond}), this reduces  using (\ref{sumofexps}) to 
$$T_N f = \sum_{k,m,n} \,  \sum_{ad=N, a,d>0} \Big(\frac{a}{d}\Big)^k  a^{w-1}   a^{(k)}_{m,n}  \Lef^k q^{\frac{ma}{d}} \overline{q}^{\frac{na}{d}}\ .$$
Replace $m$, $n$ with $m'= m/d$ and $n'=n/d$ to obtain
$$T_N f = \sum_{k,m',n'} \,  \sum_{ad=N, a,d>0} \Big(\frac{a}{d}\Big)^k  a^{w-1}   a^{(k)}_{m'\!d,n'\!d}  \Lef^k q^{\,m'\!a} \overline{q}^{\,n'\!a}\ .$$
Comparing with (\ref{corTNfexp}) and collecting terms in $q^{\mu}\overline{q}^{\nu}$ gives
$$   \alpha^{(k)}_{\mu, \nu} =  \sum_{a |(N, \mu, \nu), a\geq 1}  a^{w-1}\Big(\frac{a}{d}\Big)^k a^{(k)}_{\frac{\mu d}{a}, \frac{\nu d}{a}}  \ .
$$
where in the sum, $d$ denotes $N/a$. 
\end{proof}
Condition (\ref{Heckecond})  holds in particular if  $a^{(k)}_{m,n}=0$  for all $mn\neq 0$. 
\begin{cor} The Hecke algebra acts on $\HM^!$.
\end{cor}
 If $f  = f^a + f^0 +  f^h$ as in (\ref{Fdecompah0})
then $(T_N f )^{\bullet}= T_N (f^{\bullet})$ for $\bullet \in \{a, 0, h\}$.  It follows from the  formula that if $f^{\bullet}$ has a pole of order at most $p$
at the cusp, then $T_N f^{\bullet}$ has a pole of order at most $Np$ at the cusp, for $\bullet = a, h$.

\begin{cor} Let $f$ be as in corollary \ref{corfHecke}. Let $w=r+s$. Then
\begin{equation} \label{TNf0}   \alpha^{(k)}_{0,0} =     \sigma_{2k+w-1}(N) N^{-k}   \, a^{(k)}_{0,0}      \ . 
\end{equation} 
\end{cor}

\begin{cor}
Let $N=p$ be prime. Then for all $k, \mu, \nu$, 
$$\alpha^{(k)}_{\mu, \nu}  =   p^{-k}  a^{(k)}_{\mu p ,\nu p }  + p^{w+k-1}   a^{(k)}_{\mu/p ,\nu/p }$$
where the second term arises only if $p$ divides $\mu$ and $\nu$, and is absent otherwise.
\end{cor}

The space of almost weakly holomorphic modular forms $M^![\GE^*_2, \Lef^{\pm}]$
 are harmonic, and are preserved by the Hecke operators. 
\begin{example} 
The modified Eisenstein series 
$\GE^*_2 =  \GE_2 - {1 \over 4 \Lef} $ 
is modular of weights $(2,0)$ and lies in $\mathcal{M}_{2,0}$, where  
$$\GE_2 (q) = - \frac{1}{24} + \sum_{n=1}^{\infty} \sigma_1(n) q^n = - \frac{1}{24} + q +3 q^2 + 4q^3 + 7 q^4 + \ldots$$
By formula (\ref{TNf0}) with $w=2$, we find that $T_n(\Lef^{-1}) = n^{-1} \sigma_{-1}(n) \Lef^{-1} = \sigma_1(n) \Lef^{-1}$, and hence $\GE_2^*$ is a Hecke eigenform.   For all $n\geq 1$,
$$T_n \GE^*_2 = \sigma_1(n) \GE^*_2\ .$$
\end{example}

\begin{rem}
The quotient $\HM^! / \partial (\HM^!)$ also admits an  action of the Hecke algebra. 
\end{rem}

\subsection{Hecke operators on weakly holomorphic modular forms} 
Let $f\in M_{k+2}^!$ be a weak Hecke eigenform. Then for all $m$,
$$ (T_m - \lambda_m) f    = \psi_m $$
for some $\lambda_m$, where $\psi_m$ is  a weakly holomorphic modular form  
$$\psi_m  \quad \in \quad D^{k+1} M^!_{-k}  \ .$$
Since the operators $T_m, T_n$ commute, they satisfy 
\begin{equation} \label{TmandTnonpsi} (T_m -\lambda_m )\psi_n = (T_n - \lambda_n) \psi_m\end{equation} 
for all $m,n$.  For all $(m,n)$ coprime:
$$\psi_{mn} = \lambda_n \psi_m + T_m\, \psi_n$$
and for all $p$ prime and $n\geq 1$,  
$$\psi_{p^{n+1}} = T_p \psi_{p^n}  - p^{k+1} \psi_{p^{n-1}} + \lambda_{p^n} \psi_p \ .$$
 
 \subsection{Hecke action on modular primitives} \label{sectHeckeaction}
 Let $f, g, X_{r,s}$ be as in proposition \ref{propweaklift}.

\begin{prop}  \label{propHeckeonXrs}
 $f$ is a weak Hecke eigenform with eigenvalues $\lambda_m$ if and only if  $g$ is a weak Hecke eigenform with eigenvalues $\lambda_m$. In this case,
\begin{equation} \label{TmonXrs} \Big( T_m - {\lambda_m \over m} \Big) X_{r,s} = \frac{1}{m}\Lef^{-n} \Big(  \frac{\partial^r}{r!} \psi_m +  \frac{\overline{\partial}^s}{s!}  \overline{\phi_m} \Big) \end{equation}
for some weakly holomorphic functions $\psi_m, \phi_m \in M^!_{-n}$ satisfying
\begin{eqnarray} \label{psiandphi}
(T_m - \lambda_m ) f  &=  &    \frac{1}{n!} \, \Lef^{-n-1} \partial^{n+1} \psi_m  \\
(T_m - \lambda_m ) g  &=  &    \frac{1}{n!} \, \Lef^{-n-1} \partial^{n+1} \phi_m \ . \nonumber 
\end{eqnarray} 
\end{prop} 
 
 \begin{proof}
 Suppose that $f$ is a weak Hecke eigenform with eigenvalues $\lambda_m$. Therefore, by (\ref{IDforBol}), there exist
 for every $m\geq 1$ a $\psi_m \in M^!_{-n}$ such that
  $$ (T_m - \lambda_m ) f =    \frac{1}{n!} \, \Lef^{-n-1} \partial^{n+1} \psi_m\  .$$
 Since $\partial X_{n,0} = \Lef f$, it follows from  $[T_m , \partial]=0$ (lemma \ref{lemTmprop}) that
 $$\partial \Big(T_m  - \frac{\lambda_m}{m}\Big) X_{n,0} =  \frac{1}{m n!}  \LL^{-n}   \partial^{n+1} \psi_m  $$
Hence by lemma \ref{lemker} there exists a $\phi_m \in M^!_{-n}$ such that 
  \begin{equation}\label{inproofTmXn0} \Big(T_m  - \frac{\lambda_m}{m}\Big) X_{n,0} =  \frac{1}{m}\LL^{-n} \Big( \frac{ \partial^{n}}{n!}   \psi_m+ \overline{\phi_m}  \Big)\ . \end{equation}
  This proves the case $(r,s)=(n,0)$ of $(\ref{TmonXrs})$. By taking the complex conjugate of lemma \ref{lempartialiterate} we find, using the fact that  $\overline{\partial} \psi_m=0$, that
  $$\overline{\partial}  \partial^k   \psi_m  =   \overline{\partial}  \partial^k   \psi_m - \partial^k  \overline{\partial} \psi_m  =  k (n-k+1)   \partial^{k-1} \psi_m\ .$$
By induction on $s$, this in turn implies that
$$\frac{\overline{\partial}^s}{s!} \frac{\partial^n}{n!}\psi_m = \frac{\partial^{n-s}}{(n-s)!} \psi_m\ .$$
From the definition $X_{r,s} =\frac{\overline{\partial}^s }{s!} X_{n,0}$, we apply $\frac{\overline{\partial}^s}{s!}$ to (\ref{inproofTmXn0}) and use  $[T_m , \overline{\partial}]=0$ (lemma \ref{lemTmprop}) and the previous equation  to deduce that 
 $$\Big(T_m  - \frac{\lambda_m}{m}\Big) X_{r,s} = \frac{1}{m} \LL^{-n} \Big( \frac{ \partial^{r}}{r!}   \psi_m+  \frac{\overline{\partial}^s }{s!}\overline{\phi_m}  \Big)\  \ . $$
This proves $(\ref{TmonXrs})$. Now apply $\overline{\partial}$ to this expression in the case $(r,s)=(0,n)$. We find, since $\overline{\partial} X_{0,n} =\Lef \overline{g}$  and $\overline{\partial} \psi_m=0$ that,
$$ \Big(T_m -\frac{\lambda_m}{m} \Big) \Lef \overline{g} = \frac{1}{m}\Lef^{-n} \frac{\overline{\partial}^{n+1} }{n!} \overline{\phi_m}$$
which is equivalent by lemma \ref{lemTmprop} to the second line of $(\ref{psiandphi})$.  By (\ref{IDforBol}),  $g$ is a weak Hecke eigenform with eigenvalues $\lambda_m$, and completes the proof.  The converse result, where we assume that $g$ is a weak Hecke eigenform and deduce the same for $f$, holds by complex conjugation. \end{proof}

\begin{rem} Remark \ref{remXassection}  
  implies an equality on the Betti image  under $\mathrm{comp}_{B,dR}$ of the de Rham cohomology classes in $H^1_{dR}(\mathcal{M}_{1,1}(\C); \V_n)$:
$$[  2\pi i  f(z) (X- z Y)^{n} dz]  =  \overline{[ 2\pi i  \, g(z) (X-  z  Y)^n d z ]}\ .$$ 
Since the Hecke operators act on cohomology, it follows that $f$ is a weak Hecke eigenform if and only if $g$ is, and that they have the same eigenvalues.
Incidentally, this argument also proves that $g = \s(f)$.  \end{rem}

 \subsection{Determination of  the coefficient of $\Lef^{-n}$} \label{sectalpha}
 
 \begin{cor} Let $f , g \in   M^!_{n+2}$  and $X_{r,s}$ be as in the previous proposition. 
Then  $$X^0_{r,s}  =  \alpha (-1)^{r}  \binom{n}{r} \Lef^{-n}  \ , $$
 where the constant $\alpha \in \C$ satisfies for all $m \geq 1$
 $$  (\sigma_{n+1}(m) - \lambda_m ) \, \alpha =   a_{0}(\psi_m) +  \overline{a_0(\phi_m)}  $$
  where $a_0$ denotes the zeroth Fourier coefficient.
  
   If $f$ is cuspidal,     $ \sigma_{n+1}(m)-\lambda_m \neq 0$ for all $m$ sufficiently large, in which case
  \begin{equation} \label{alphaformula} \alpha = \frac{  a_{0}(\psi_m) +  \overline{a_0(\phi_m)}  }{  \sigma_{n+1}(m) - \lambda_m }  \ .
  \end{equation}
 
    \end{cor} 
 
 \begin{proof}
 By proposition \ref{propHeckeonXrs}, 
 $$( m T_m - \lambda_m)  X_{n,0} =  \Lef^{-n} \Big(\frac{\partial^n}{n!} \psi_m + \overline{\phi_m}\Big) $$
for all $m\geq 1$.  By equation (\ref{Xrsexplicitformula}), the constant term of $X_{n,0}$ is
$X^0_{n,0} = \alpha \, \Lef^{-n}$, 
since $f$ is a cusp form. 
On the other hand, equation (\ref{TNf0}) in weight $n$ implies that
 $$m  T_m\,  \Lef^{-n} = \sigma_{-n-1}(m) m^{n+1} \,  \Lef^{-n} = \sum_{d|m} \Big(\frac{m}{d}\Big)^{n+1} \Lef^{-n} = \sigma_{n+1}(m)\,  \Lef^{-n}\ .$$
 Putting the pieces together yields 
 $$ (m T_m - \lambda_m) X_{n,0}^0 =  (  \sigma_{n+1}(m) - \lambda_m) \alpha \,\Lef^{-n} =  \Lef^{-n}\,  \big(  a^{(0)}_{0,0} \Big( \frac{\partial^n}{n!} \psi_m\Big)+  a^{(0)}_{0,0} (  \overline{\phi_m})\big)  $$
 where $a^{(0)}_{0,0}$ denotes the  coefficient in the expansion (\ref{fexp}). Since $\psi_m \in M_{-n}^!$, 
 $$a^{(0)}_{0,0}( \partial^n \psi_m) = a^{(0)}_{0,0}(  \partial_{-1} \ldots \partial_{-n} \psi_m) =  (-1)^n n! a^{(0)}_{0,0} (\psi_m)$$
 by successive application of  (\ref{partialraction}), which never decreases the powers of $\Lef$. 
 
The $\lambda_m$ are the eigenvalues of a normalised  holomorphic Hecke eigenform $g \in S_{n+2}$.  Then $\lambda_m = a_m(g)$ and  an elementary estimate \cite{Lang}, Lemma 2, implies that   $|a_m(g)|$ grows at most like $m^{n/2+1}$. 
Since $\sigma_{n+1}(m)\geq m^{n+1}$, it follows that $( \sigma_{n+1}(m)- \lambda_m)$ is non-zero for sufficiently large $m$.
 \end{proof} 

 The consistency of equations (\ref{alphaformula})   for different values of $m$ follows from (\ref{TmandTnonpsi}). Equation (\ref{alphanotation}) would have poles for every $n$ if $f$ were an Eisenstein series  by (\ref{TmonGE}).

 \section{Existence of modular primitives} \label{sectExistence}
Having determined the form of modular primitives, we now turn to their existence.

 \subsection{Cocycles and periods}
 Let us fix  a system  $\llam$  of Hecke eigenvalues corresponding to a cuspidal eigenform $f\in S_{n+2}$,  and let $H^{dR}_{\llam}$, $H^B_{\llam}$ be as defined in \S\ref{sectWeakHol}. For simplicity, we shall drop the subscripts  $\llam$ from now on, and set $K= K_{\llam}$.  Let 
 $$f  \in M_{n+2}  \qquad \hbox{ and } \qquad  f' \in M^!_{n+2}$$
 denote a $K$-basis  for $H^{dR}$  of the form  (\ref{Hbasis}). Likewise, choose a $K$-basis $P^+$ of $H^{B,+}$ and $P^-$ of $H^{B,-}$. We have $P^{\pm}_T =0$. The polynomials  $P^{\pm}_S$ are known, respectively,  as the even  and odd period polynomials of $f$.

 Let us choose a basepoint $z_0 \in \HH$ and let  $C, C' \in Z^1(\Gamma; V_n)$ denote the cocycles associated to $f$ and  $f'$, respectively.  
 The comparison isomorphism (\ref{compisom})  implies that 
 \begin{eqnarray} \label{classesCfasPs}
{[}C']   & = &   \eta^+  \,  [P^+]  + i \eta^- \, [P^-]   \\
{[}C]   & = &   \omega^+  \, [P^+]  + i \omega^- \, [P^-]  \nonumber
\end{eqnarray} 
where $\omega^{+},  i \omega^{+}, \eta^{+}, i \eta^{-}$ are the entries of the period matrix  in these bases. 
\begin{lem} There exists a canonical Hecke-equivariant splitting  over $\Q$:
$$s: H_{\mathrm{cusp}}^1(\Gamma; V_n) \To Z_{\mathrm{cusp}}^1(\Gamma; V_n)\ .$$
\end{lem} 
\begin{proof} See \cite{MMV}, lemma 7.3. \end{proof}
We can assume that $P^+, P^- \in Z_{\mathrm{cusp}}^1(\Gamma; V_n \otimes K)$ are the unique Hecke-equivariant lifts of the cohomology classes chosen earlier.  They satisfy $P^{\pm}_T=0$. 

\begin{cor} \label{corcocyff'} 
There exist polynomials $Q, Q' \in V_n \otimes \C$ such that  for all $\gamma \in \Gamma$, 
\begin{eqnarray}
C'_{\gamma} & = &  \eta^+  \, P^+_{\gamma}  + i \eta^- \, P^{-}_{\gamma}    +  Q'(X,Y)\big|_{\gamma -\id}   \nonumber \\
C_{\gamma} & = &  \omega^+  \, P^+_{\gamma}  + i \omega^- \, P^{-}_{\gamma}    +  Q(X,Y)\big|_{\gamma -\id}   \ . \nonumber 
\end{eqnarray} 
The polynomials $Q, Q'$   depend on the choice of basepoint $z_0$. \end{cor}

\subsection{Real and imaginary analytic cusp forms} We shall construct explicit modular primitives of cusp forms in two steps. 
Recall that  the integrals $F_{f}(\tau)$ were defined in (\ref{Ffdefn}) relative to the basepoint $z_0 \in \HH$. 
 
\begin{defn} Define  real analytic functions $\HH \rightarrow V_{n} \otimes \C$ by 
\begin{eqnarray}
\mathcal{I}_f(\tau) & = & (2 \pi i)^{-2n} \Big(  \omega^+ \mathrm{Re}\, \big(  \mathcal{F}_{f'}(\tau)  -Q' \big)  - \eta^+   \mathrm{Re} \, \big( \mathcal{F}_{f}(\tau)  -Q\big) \Big)\nonumber  \\ 
\mathcal{R}_f(\tau) & = &   (2 \pi i)^{-2n} \Big(  \omega^{-} \mathrm{Im}\, \big(  \mathcal{F}_{f'}(\tau)  -Q'   \big)  - \eta^{-}   \mathrm{Im} \, \big( \mathcal{F}_{f}(\tau) - Q \big) \Big)    \nonumber \ .
\end{eqnarray} 
Note that (\ref{Ffdefn}) involves an odd power of $2 \pi i$, which explains why  `real' and `imaginary' are apparently  interchanged. 
\end{defn}
These functions satisfy  the differential equations
\begin{eqnarray}  d \mathcal{I}_f(z) & = &  \omega^+ \mathrm{Re}\, \big( 2\pi i  f'(z) (X- z Y)^n dz \big)  -\eta^+ \mathrm{Re}\, \big( 2\pi i  f(z) (X- z Y)^n dz \big) 
\nonumber \\
  & = &  \pi i \, \big( \omega^{+}  f'  - \eta^{+} f  \big) (X-zY)^n dz +  \pi i \big( \eta^{+} \overline{f}- \omega^{+}  \overline{f'} \big)  (X-  \overline{z} Y)^n d\overline{z}  \nonumber 
\end{eqnarray}
and  similarly
$$ d \mathcal{R}_f(z) =   \pi i \, \big( \omega^{-}  f'  - \eta^{-} f  \big) (X-zY)^n dz +  \pi i \big(\omega^{-}  \overline{f'}-  \eta^{-} \overline{f}  \big)  (X-  \overline{z} Y)^n d\overline{z}  $$

\begin{thm}
The functions $\mathcal{I}_f(\tau)$ and $\mathcal{R}_f(\tau)$ are well-defined (independent of the choice of basepoint $z_0$), and $\Gamma$-equivariant.
\end{thm}
 \begin{proof}  The $\Gamma$-equivariance of    $\mathcal{I}_f(\tau)$ follows from  corollary \ref{corcocyff'}:
 $$\mathcal{I}_f(\gamma \tau) \big|_{\gamma}- \mathcal{I}_f(\tau) =   \omega^+     \big(  \eta^+  \, P^+_{\gamma}  )  - \eta^+   \big( \omega^+  \, P^+_{\gamma}    \big)  = 0 \ .$$
  Changing base point  $z_0$ yields a modular equivariant solution to the same differential equation for $d \mathcal{I}_f$ given above (which is independent of the basepoint). By lemma \ref{lemEsection}, any  modular equivariant solution is unique.   The argument for $\mathcal{R}_f(\tau)$ is similar. 
  \end{proof}

Extract the coefficients of $\mathcal{I}_f$ and $\mathcal{R}_f$ via 
\begin{eqnarray} \mathcal{I}_f  &= &  \sum_{r+s=n} I_{r,s} (X-zY)^r (X-\overline{z} Y)^s \nonumber \\
 \mathcal{R}_f  &= &  \sum_{r+s=n} R_{r,s} (X-zY)^r (X-\overline{z}Y)^s \nonumber 
\end{eqnarray} 
  They 
define weakly holomorphic modular forms in $\M^!$.
 \begin{cor} There exists a family $I_{r,s} \in  \M^!_{r,s}$, for $r+s=n$,  such that 
 \begin{eqnarray} \label{Rfamilyeqn1} 
 \partial \, I_{r,s} & = &  (r+1) I_{r+1, s-1}    \qquad \qquad  \hbox{ for all  } \, 1\leq s\leq n \  \nonumber  \\
 \overline{\partial} \,I_{r,s} & = & (s+1) I_{r-1, s+1}  \qquad  \qquad \hbox{ for all  } \, 1\leq r \leq n \ , \   \nonumber 
\end{eqnarray} 
and
$$ \partial \, I_{n,0}   =  \LL   \big(  \omega^+  f'     - \eta^+  f  \big)    \qquad , \qquad   \overline{\partial} \, I_{0,n}   =  \LL   \big(   \eta^+  \overline{f} - \omega^+  \overline{f'}     \big) $$
They are `imaginary' in the sense that  $\overline{I}_{r,s}= -I_{s,r}$.

Similarly,  there exists a family of elements  $R_{r,s} \in   \M^!_{r,s}$  for $r+s=n$ and $r,s\geq 0$, 
 satisfying the identical equations, except that the last line is replaced by 
 $$  \partial \, R_{n,0}   =  \LL   \big(  \omega^{-}  f'     - \eta^{-}  f  \big)    \qquad \hbox{ and } \qquad 
 \overline{\partial} \, R_{n,0}   =    \LL   \big(  \omega^{-}  \overline{f'}  - \eta^{-}  \overline{f}     \big)  $$
They are `real' in the sense that  $\overline{R}_{r,s} =  R_{s,r}$.  
 \end{cor} 
 \begin{proof} This is a straightforward application of (\ref{partialFgeneral}) to the previous discussion.
   \end{proof}

   \subsection{Modular primitives of cusp forms}
   Since the period isomorphism is invertible, we can  change basis, to deduce the existence of modular primitives for all cusp forms.
   
   \begin{defn} 
 For any  basis $f,f'$ of (\ref{Hbasis})   define
 \begin{eqnarray}  \mathcal{H}(f) &=  &   p^{-1}\big( \omega_f^{-} \mathcal{I}_f - \omega_f^+ \mathcal{R}_f\big) \nonumber \\
 \mathcal{H}(f')&= &  p^{-1} \big(  \eta^-_f  \mathcal{I}_f - \eta_f^+  \mathcal{R}_f   \big)\nonumber  
 \end{eqnarray} 
 where $p=  ( \omega_f^+ \eta_f^{-} - \omega_f^{-} \eta_f^{+} )=   - i  \, \det(P_f)\neq 0  $. 
 \end{defn}
 Write $\mathcal{H}(f)= \sum_{r+s=n} \mathcal{H}(f)_{r,s} (X-zY)^r(X-\overline{z}Y)^s$ as usual. 
 
 \begin{thm}  \label{thmHfrs} The family of functions $\mathcal{H}(f)_{r,s}$ satisfy the equations
 \begin{eqnarray} 
\partial \, \mathcal{H}(f)_{r,s} & = & (r+1) \mathcal{H}(f)_{r+1, s-1}       \qquad  \hbox{ for all  } \, 1\leq s\leq n \  \nonumber  \\
 \overline{\partial} \, \mathcal{H}(f)_{r,s} & = & (s+1) \mathcal{H}(f)_{r-1, s+1}    \qquad  \hbox{ for all  }\, 1\leq r\leq n \  \nonumber 
\end{eqnarray} 
and
$$\partial \, \mathcal{H}(f)_{n,0}   =  \LL  f  \nonumber   \qquad , \qquad   \overline{\partial} \, \mathcal{H}(f)_{0,n}   = \LL \overline{\s(f)}   $$
The family of functions $\mathcal{H}(f')_{r,s}$  satisfy the same equations with $f$ interchanged everywhere with  $f'$, and $\omega$ interchanged with $\eta$. 
\end{thm}

\begin{proof} Straightforward consequence of  the previous corollary using:
$$\s \big(  \omega^+  f'     - \eta^+  f  \big) = -  \big(  \omega^+  f'     - \eta^+  f  \big) \qquad \hbox{ and } \qquad \s \big(  \omega^{-}  f'     - \eta^{-}  f  \big) = \big(  \omega^{-}  f'     - \eta^{-}  f  \big)\ . $$
\end{proof}
It follows from  uniqueness (lemma \ref{lemEsection}) that $\mathcal{H}(f)_{r,s}$ is well-defined (only depends on $f$ and not the choice of basis $f, f'$), since it only depends on $f$ and its image under the single-valued involution $\s(f)$, which is canonical.  

In this manner we have defined a canonical modular primitive
\begin{eqnarray} H^{dR}_{\llam} & \To  &  \mathcal{M}^!_{n,0} ( K_{\llam} [ \omega_{\llam}^{\pm}, \eta_{\llam}^{\pm}]) \  . \nonumber \\
 f & \mapsto &   \mathcal{H}(f)_{n,0} \nonumber 
 \end{eqnarray} 
This map is injective since $\Lef^{-1} \partial   \mathcal{H}(f)_{n,0}  = f$. 

\begin{cor} For all $r+s =n$, complex conjugation acts via:
$$\overline{\mathcal{H}(f)}_{r,s} = \mathcal{H}(\s(f))_{s,r}\ .$$
\end{cor} 

\begin{cor}   Suppose that $f \in S_n$ is a cuspidal  Hecke eigenform. The constant term in $\mathcal{H}(f)_{r,s}$ is
 proportional to the Petersson norm of $f$ times $\Lef^{-n}$
$$\mathcal{H}(f)^0 \in   \{f, \s(f) \}\, K_{\llam} \Lef^{-n}\ .$$
\end{cor} 
\begin{proof}   This follows from (\ref{alphaformula}) since the coefficient of $f'$ in $\s(f)$ is proportional to $\{f,\s(f)\}$. The latter can be interpreted as the Petersson norm via  (\ref{sff})  and the comments which follow.
 \end{proof}

\begin{cor} Every modular form  admits a modular primitive  in $\mathcal{M}^!$.
   \end{cor} 
   \begin{proof} Every modular form of integral weight is a linear combination of Eisenstein series and cuspidal  Hecke eigenforms. 
   \end{proof}

\subsection{Vanishing constant term}
The space  $H^{dR} \otimes \C$ decomposes into eigenspaces with respect to the map $\s$:
$$H^{dR}\otimes \C =   (H^{dR}\otimes \C)^+ \oplus (H^{dR}\otimes \C )^- \ .$$
They are respectively the preimages of $H_B^{\pm} \otimes \C$ under the comparison isomorphism. 

An element $f \in  (H^{dR}\otimes \C)^+$ satisfies  $\s(f) = f$, and  hence $\mathcal{H}(f)_{r,s}$ is proportional to the `real' function $\mathcal{R}(f)_{r,s}$. 

An element $f \in  (H^{dR}\otimes \C)^-$ satisfies  $\s(f) = -  f$, and  hence $\mathcal{H}(f)_{r,s}$ is proportional to the 
 `imaginary' function  $\mathcal{I}(f)_{r,s}$.  The latter satisfies $\mathcal{I}_{r,s}^0 =0 $  by lemma \ref{lemEsection} and has no constant part.  It is therefore cuspidal: $\mathcal{I}(f)_{r,s } \in \mathcal{S}^!$.

\subsection{Remark on Galois action}
The Galois group $\mathrm{Gal}(\overline{\Q}/ \Q)$ acts on the space $M^! (\overline{\Q})$ via  its action on  coefficients. 
In particular, it preserves the  subspace  of functions $f \in M^!_{n}(\overline{\Q})$  such that $\mathrm{ord}_{\infty}(f) \geq - \dim S_n$.
 It follows that the decomposition 
 $$M^!_{n+2}   =  D^{n+1} M^!_{-n} \oplus M^!_{n+2}/D^{n+1} M^!_{-n} $$
 is $\mathrm{Gal}(\overline{\Q}/ \Q)$-equivariant. 
 Since $\mathcal{H}$ is injective, we obtain a naive action of the Galois group via $ \sigma \,\mathcal{H}( f)_{r,s}:= \mathcal{H}(\sigma f)_{r,s}$ for all $\sigma \in \mathrm{Gal}(\overline{\Q}/ \Q)$.    Since the coefficients of $\mathcal{H}_{r,s}$ are not algebraic, this is not induced by an action on Fourier coefficients.

 This could be promoted to an action on Fourier coefficients by the following sleight of hand. Since $H^1_{dR}(\mathcal{M}_{1,1}; \V_n)$ is the de Rham realisation of a pure motive (it would be enough to consider it as an object in a Tannakian category of realisations) it admits an action of a (reductive) algebraic group scheme $G$ over $\Q$. The group $G$ acts  on the  corresponding motivic periods \cite{NotesMot}.  Extending scalars gives
 $$H^1_{dR}(\mathcal{M}_{1,1}; \V_n) \otimes_{\Q}  \overline{\Q} = \bigoplus_{\llam}  H^{dR}_{\llam}\otimes_{K_{\llam}} \overline{\Q}$$
 and so $G\times_{\Q} \overline{\Q}$ is an extension of the quotient   of $\mathrm{Gal}(\overline{\Q}/\Q)$ which acts by permuting the factors,  by a group $H$ contained in the subgroup of $\prod_{\llam} \mathrm{GL}( H^{dR}_{\llam})$ of automorphisms with the same determinant on each factor.  A choice of basis of remark \ref{remcanbasis} gives a splitting of this extension (since the conditions $a_1(f)=1$ or $\{f',f'\}=0$ are  Galois-equivariant). 
 We can  define  modular lifts $\mathcal{H}(f)^{\mm}_{r,s}$ by  replacing periods with motivic periods $\omega^{\mm, \pm}_f$, $\eta^{\mm, \pm}_f$. The group $G(\overline{\Q})$ will act upon the   $\mathcal{H}(f)^{\mm}_{r,s}$ via its action on  the coefficients in the expansion  $\mathcal{H}(f)^{\mm}_{r,s}$. This  gives back the naive action of $\mathrm{Gal}(\overline{\Q}/\Q)$ above (which is independent of the choice of splitting).

\section{Example: Real analytic version of Ramanujan's function $\Delta$} \label{sectExample}

\subsection{Weakly holomorphic cusp forms in weight $12$}
Let $n=10$. Let $\Delta$ denote Ramanujan's cusp form of weight 12
 $$\Delta = q\prod_{n\geq 1}(1-q^n)^{24}= q  -24 \, q^2 +252\, q^3+  1472\, q^4  + 4830 q^5 -6048q^6+ \ldots $$
 Since $\dim S_{12} =1$, it is a Hecke eigenform with eigenvalues  in $\Z$. 
 There exists a unique weakly holomorphic modular form $\Delta' \in M^!_{12}$ which has a pole of order at most $1$ at the cusp, and whose Fourier coefficients $a_0$, $a_1$ vanish. Explicitly,  
 $$ \Delta'= q^{-1} +47709536\, q^2+39862705122\, q^3+7552626810624\, q^4+ \ldots$$
 It satisfies $\{\Delta, \Delta'\}=1$. 
 It follows  that there is a single cuspidal Hecke eigenspace, and that it has the  de Rham basis:
 $$H^{dR} = H^1_{\mathrm{cusp}, dR}(\mathcal{M}_{1,1}; \V_{10}) = \Q \Delta' \oplus \Q \Delta\ .$$
 The function $\Delta'$ is a weak Hecke eigenform with the same eigenvalues as $\Delta$: 
 $$(T_m - \lambda_m )\Delta' = D^{11} p_{m} \qquad \hbox{ for all } m \geq 1 $$
 for some $p_m \in M_{10}^!$.
 For example,  $\lambda_2 = -24$ and hence
 $(T_2  -\lambda_2) \Delta' =   D^{11} \, \psi_2$, 
 where 
 $$ p_2 = 24 \, \GE_{14} \Delta^{-2} = -q^{-2}-24 q^{-1} +196560+47709536 q + \ldots  \ .$$
In the notations of proposition \ref{propHeckeonXrs}, we have $\psi_2 = 10! \,2^{-11} p_2$ by (\ref{IDforBol}).
 
 \subsection{Cocycles}
Let $P^{\pm} \in Z^1(\Gamma; V_{10})$  be the Hecke-invariant cocycles $P^{\pm}: \Gamma \rightarrow V_{10}$  which are uniquely determined by 
$P^{\pm}_T=0$ and 
 \begin{eqnarray} 
 P^+_S   &= &   \frac{36}{691} (Y^{10}-X^{10}) +X^2 Y^2 (X^2-Y^2)^3  \nonumber \\ 
 P^-_S & = & 4 X^9 Y-25 X^7 Y^3+42 X^5 Y^5-25 X^3 Y^7+4 X Y^9 \nonumber 
 \end{eqnarray} 
 Their Haberlund inner product is $\{P^+, P^-\}=1$.  They provide a Betti basis
 $$H^B=  H^1_{\mathrm{cusp}, B}(\mathcal{M}_{1,1}; \V_{10}) = \Q P^+ \oplus \Q P^-\ .$$
 
 \subsection{Periods}
 
 Following the method given in \S\ref{sectWeakHol}, we can easily compute the period matrix (\ref{periodmatrix}) in this basis.  We find that for all $\gamma \in \Gamma$, 
 $$\int_{\gamma}  (2\pi i)^{11} \Delta(z) (X-zY)^{10} dz =  \omega^+ P^+_{\gamma} + \omega^- P^-{\gamma}$$
 where  
 $$\omega_+ =  -68916772.809595194754\ldots \quad , \quad   \omega_- =  -5585015.3793104018668\ldots  $$
which agree with the numerical values for the periods of $\Delta$ given in the literature. The periods of $\Delta'$, on the other hand,  are $\eta_+ , i \eta_-$ where 
 $$\eta_+  =   127202100647.17709477 \ldots \quad , \quad  \eta_- = 10276732343.649132750 \ldots$$
 I could find no reference for these  values for comparison. 
 In accordance with proposition 5.6 of \cite{BH}, we can indeed verify numerically that
 $$\det \begin{pmatrix} \eta_+ & \omega_+ \\  i \eta_- &  i \omega_- \end{pmatrix} =  10!  \times (2 \pi i)^{11} \ . $$
 The Petersson norm  of $\Delta$, in its standard normalisation,  is
 $$  \frac{ - 2 \omega_+ \omega_{-}}{2^{11} (2\pi i)^{22}} = 0.00000103536205 \ldots  >0  $$

\subsection{Single-valued involution}
The single-valued period matrix is 
$$ \frac{i}{10!   (2 \pi i)^{11} }  \begin{pmatrix} \eta^+ \omega^- + \eta^- \omega+ & 2 \omega^+ \omega^- \\ 
2 \eta^+ \eta^-  & -( \eta^+ \omega^- + \eta^- \omega+)  \end{pmatrix}   
= \begin{pmatrix} 648.84093\ldots  & -0.3520770\ldots \\
 1195742.7\ldots &  -648.84093\ldots \end{pmatrix} $$
in the basis $\Delta, \Delta'$. It does not depend on the choice of Betti basis. 
Therefore 
$$\s(\Delta) =   \sigma  \Delta'   +    \tau \Delta  \qquad \hbox{ where } \sigma =  -0.35207\ldots \quad , \quad \tau=  -648.84\ldots$$
For convenience, we evaluate the ratio
$$  \rho = \frac{\tau}{\sigma} = - \frac{ \eta^+ \omega^- + \eta^- \omega^+}{2 \omega^+ \omega^-} = -\frac{1}{2} \Big( \frac{\eta^+}{\omega^+} +  \frac{\eta^-}{\omega^-}\Big) = 1842.8947269\ldots  $$

 \subsection{The constant term}
Since $\sigma_{11}(2) = 2049$, we check that
$$\sigma_{11}(2) - \lambda_2 = 2073 = 3. 691$$
is non-zero and therefore since $\psi_2 = 10!\, 2^{-11} p_2$,  
$$\frac{a_0 (\psi_2)}{\sigma_{11}(2) - \lambda_2} =  \frac{10!}{2^{11}} \, \frac{196560}{3. 691} =  \frac{10!}{2^{11}} \frac{\, 7! \, 13}{ 691} \ .$$
The $691$ in the denominator is a consequence of the congruence
$\Delta \equiv  \GE_{12} \pmod{691}$. 
Formula (\ref{alphaformula})  therefore that  implies that 
$$\alpha =  \frac{\, 7! \, 13}{ 691} \frac{10!}{2^{11}} \,  \sigma    =   \frac{\, 7! \, 13}{ 691}   \frac{2 i \omega^+ \omega^-}{ (4\pi i)^{11}} \ . $$

\subsection{Real analytic cusp form} The  real analytic cusp forms $\mathcal{H}(\Delta)_{r,s}$ for $r+s =10$ can be written down explicitly from the formulae given in the introduction.

\subsection{The mock modular form $M_{\Delta}$}
Denote the Fourier coefficients of $\Delta, \Delta'$ by $a_n, a_n'$.  Our formula for  the `mock' modular form is
$$M_{\Delta} = \alpha + \frac{10!}{2^{11}} \sum_{n}  \frac{\sigma a'_n + \tau a_n}{n^{11}} q^n$$
where $\sigma, \tau$ are the periods given above.
In order to compare with Ono's formula, let us rescale by setting 
$$M'_{\Delta} = - \frac{11 \times  2^{11} }{\sigma}  M_{\Delta} =   11! \Big(    - \frac{\, 7! \, 13}{ 691}  +    \sum_{n}  \frac{ a'_n + \rho a_n}{n^{11}}q^n \Big)\ .$$
Its first five Fourier coefficients are given exactly by 
$$ 11! \Big(  q^{-1}-\frac{65520}{691}-\rho q+\Big(\frac{3}{256}\rho-\frac{1490923}{64}\Big) q^2+ \Big(-\frac{28}{19683}\rho-\frac{164044054}{729}\Big)q^3+ \ldots 
\Big)  \ .$$
 By uniqueness, this function coincides with the mock modular form for $\Delta$ given in \cite{Ono} and discussed in \cite{BOR}. We have verified, by substituting the above numerical value of $\rho$, that this agrees with the  computation in \cite{Ono} to the accuracy given in that paper.

Ono's formula  \cite{Ono} for its nth Fourier coefficient, for $n>0$, is:
$$ - 2 \pi \, \Gamma(12) n^{-\frac{11}{2}}\, \sum_{c=1}^{\infty} \frac{K(-1,n,c)}{c}  I_{11} \Big( \frac{4 \pi \sqrt{n}}{c}\Big)$$
where $K$ is a Kloosterman sum and $I$ is a Bessel function.
Combining this with our expression for its Fourier coefficients proves corollary \ref{corKloost}. 

\bibliographystyle{plain}
\bibliography{main}

\begin{thebibliography}{99}

\bibitem{Bol} {\bf G. Bol}:
{\it Invarianten linearer differentialgleichungen}, Abh.\ Math.\ Sem.\ Univ.\ Hamburg 16 (1949), 1--28.


\bibitem{ZagFest} {\bf F. Brown}: {\it A class of non-holomorphic modular forms I},  \url{arXiv:1707.01230}
\bibitem{EqEis} {\bf F. Brown}: {\it A class of non-holomorphic modular forms II - Equivariant iterated Eisenstein integrals},   \url{arXiv:1707.01230}


 \bibitem{MMV} {\bf F. Brown}: {\it Multiple Modular Values and the relative completion of the fundamental group of $M_{1,1}$}, \url{arXiv:1407.5167} 

\bibitem{NotesMot} {\bf F. Brown}: {\it Notes on motivic periods}, \url{arXiv:1512.06410}
 
 \bibitem{BH} {\bf F. Brown, R. Hain}: {\it Algebraic de Rham theory for weakly holomorphic modular forms of level one}, \url{arXiv:1708.00652}
 
 \bibitem{ESforMock} {\bf K. Bringmann, P. Guerzhoy, Z. Kent, K. Ono}: {\it Eichler-Shimura theory for mock modular forms},  
 Math. Ann. 355 (2013), no. 3, 1085-1121. 
 
\bibitem{BringOno} {\bf K. Bringmann, K. Ono}: {\it Lifting cusp forms to Maass forms with an application to partitions}, Proc. Natl. Acad. Sci. USA 104 (2007), no. 10, 3725-3731
 
 \bibitem{BF} {\bf J. Bruinier, J.  Funke}, {\it On two geometric theta lifts}: Duke Math. J. 125 (2004), no. 1, 45-90. 
 
  \bibitem{BOR} {\bf J. Bruinier, Jan, K.  Ono, R.  Rhoades}, {\it Differential operators for harmonic weak Maass forms and the vanishing of Hecke eigenvalues.}, Math. Ann. 342 (2008), no. 3, 673-693. 
 
 
 \bibitem{Candelori} {\bf L. Candelori}, {\it   Harmonic weak Maass forms of integral weight: a geometric approach}, Math. Ann. 360 (2014), no. 1-2, 489-517. 

\bibitem{CandeloriCastella} { \bf  L. Candelori, Castella, F.}, {\it A geometric perspective on p-adic properties of mock modular forms}, 
Res. Math. Sci. 4 (2017), Paper No. 5, 15 pp

 \bibitem{Coleman} {\bf R. Coleman}, {\it Classical and overconvergent modular forms},  J. Th. des Nombres de Bordeaux 7 (1995), no. 1,  333-365.
 
 
 
\bibitem{Guerzhoy} {\bf P. Guerzhoy}: {\it Hecke operators for weakly holomorphic modular forms and supersingular congruences}, Proc. Amer. Math. Soc. 136 (2008), no. 9, 3051-3059


\bibitem{HaGPS} {\bf R. Hain}: {\it The Hodge-de Rham theory of modular groups},  Recent advances in Hodge theory, 422-514,
London Math. Soc. Lecture Note Ser., 427, Cambridge Univ. Press, Cambridge, (2016),	\url{arXiv:1403.6443}


\bibitem{Lang} {\bf S. Lang}: {\it Introduction to modular forms}, Grundlehren der Mathematischen Wissenschaften, 222, Springer (1995)

 \bibitem{Maass} {\bf H. Maass}:{\it Lectures on modular functions of one complex variable},  Notes by Sunder Lal. Tata Institute of Fundamental Research Lectures on Mathematics, No. 29 Tata Institute (1964).

 \bibitem{Ma0} {\bf Y. Manin}: {\it Periods of parabolic points and p-adic Hecke series}, Math. Sb., 371-393 (1973). 

\bibitem{Ono} {\bf K. Ono}: {\it A mock theta function for the Delta-function},  Combinatorial number theory, 141-155, Walter de Gruyter, Berlin, (2009). 




 \bibitem{Scholl} { \bf A. Scholl}: \emph{Motives for modular forms}, Invent. Math. 100, 419-430 (1990)

\bibitem{ScholldR} {\bf A. Scholl}: \emph{Modular forms and de Rham cohomology; Atkin-Swinnerton-Dyer congruences},
Invent. Math. 79 (1985), no. 1, 49-77. 

\bibitem{SchollKazalicki} {\bf A. Scholl, M. Kazalicki}: \emph{Modular forms, de Rham cohomology and congruences.},
Trans. Amer. Math. Soc. 368 (2016), no. 10, 7097-7117. 


\bibitem{Serre} {\bf J. P. Serre}: {\it Cours d'arithm\'etique}, 4e \'edition, PUF, Paris (1995) 

 \bibitem{Shimura} {\bf G. Shimura}: {\it  On the periods of modular forms}, Math. Annalen 229 (1977), 211-221.

\bibitem{Shimura2} {\bf G. Shimura}: {\it  Sur les int\'egrales attach\'ees aux formes modulaires},  J. math. Soc. Japan 11 (1959), 291-311.





\end{thebibliography}

\end{document}